\documentclass[12pt]{article}
\usepackage{amssymb,amsmath,graphicx,amsthm,mathrsfs}
\newtheorem{Theorem}{Theorem}[section]
\newtheorem{Lemma}[Theorem]{Lemma}
\newtheorem{Proposition}[Theorem]{Proposition}

\newtheorem{Corollary}[Theorem]{Corollary}

\newtheorem{Remark}[Theorem]{Remark}
\numberwithin{equation}{section}

 \def\p{\partial} 
\def \Vh0{\stackrel{\circ}{V}_h}

\newcommand{\q}{\quad}   \newcommand{\qq}{\qquad}

\def\ms{\medskip}  \def\ss{\smallskip}

\newcommand{\lc}
{\mathrel{\raise2pt\hbox{${\mathop<\limits_{\raise1pt\hbox
{\mbox{$\sim$}}}}$}}}

\newcommand{\gc}
{\mathrel{\raise2pt\hbox{${\mathop>\limits_{\raise1pt\hbox{\mbox{$\sim$}}}}$}}}

\newcommand{\ec}
{\mathrel{\raise2pt\hbox{${\mathop=\limits_{\raise1pt\hbox{\mbox{$\sim$}}}}$}}}

\def\bb{\begin{equation}} \def\ee{\end{equation}}

\def\beqn{\begin{eqnarray}}  \def\eqn{\end{eqnarray}}

\def\beqnx{\begin{eqnarray*}} \def\eqnx{\end{eqnarray*}}

\def\bn{\begin{enumerate}} \def\en{\end{enumerate}}

\def\bd{\begin{description}} \def\ed{\end{description}}

\def \om {\Omega}

\def \p  {\partial}
\def \l  {\lambda}
\def \d  {\displaystyle}

\def \e  {\varepsilon}

\def\cN{{\cal N}}

\def\D{\Delta}

\def\deq{\mathop{\buildrel\D\over=}}
\def\bp{\begin{proposition}}
\def\ep{\end{proposition}}
\def\ba{\begin{array}}
\def\ea{\end{array}}
\def\liminf{\mathop{\underline{\rm lim}}}
\def\sqr#1#2{{\vcenter{\vbox{\hrule height.#2pt
              \hbox{\vrule width.#2pt height#1pt \kern#1pt \vrule width.#2pt}
              \hrule height.#2pt}}}}
\def\signed #1{{\unskip\nobreak\hfil\penalty50
              \hskip2em\hbox{}\nobreak\hfil#1
              \parfillskip=0pt \finalhyphendemerits=0 \par}}
\def\endpf{\signed {$\sqr69$}}
\def\see{{\it see} }

\def\ds{\displaystyle}
\def\nm{\noalign{\ms}}
\def\ns{\noalign{\ss}}
\def\d{\delta}
\def\l{\lambda}

\title{Equivalence of three different kinds of optimal control problems for heat equations and its applications}

\author{Gengsheng Wang\thanks{
 School of Mathematics and Statistics, Wuhan University, Wuhan,
430072, China. (wanggs62@yeah.net) The author was partially
supported by National Basis Research
 Program of China (973 Program) under grant 2011CB808002 and
the National Natural Science Foundation of China under grant
11161130003 and 11171264.} \q \q Yashan Xu\thanks{School of
Mathematical Sciences, Fudan University, KLMNS, Shanghai 200433,
China. (yashanxu@fudan.edu.cn) This work was partially supported by  NNSF Grant 10801041,
10831007.} }

\begin{document}

\date{}

\maketitle

\begin{abstract} This paper presents an equivalence theorem for three
different kinds of optimal control problems, which are  optimal
target control problems,  optimal norm control problems and optimal
time control problems. Controlled systems in this study are
internally controlled heat equations.  With the aid of this theorem,
we establish an optimal norm feedback law and
 build up  two algorithms   for
 optimal norms (together with optimal norm controls) and optimal
 time (along with optimal time controls), respectively.

\medskip


\end{abstract}

\section{Introduction}

$\;\;\;\;$We begin with introducing the controlled system. Let $T$
be a positive number and $\om\subseteq {\mathbb{R}}^d$ be a bounded
domain with a smooth boundary $\p \om$. Let  $\omega$ be an open and
non-empty subset of $\om$.  Write $\chi_\omega$ for the
characteristic function of $\omega$. Consider the following
controlled heat equation: \bb\label{1.1} \left\{\begin{array}{lll}
\p_ty-\triangle y=\chi_\omega \chi_{(\tau,T)}u&\mbox{in}&\om\times
(0,T),\\
y=0&\mbox{on}&\p \om\times (0,T).\\
y(0)=y_0 &\mbox{in}&\om,
\end{array}\right.\ee
Here  $y_0\in L^2(\Omega)$,  $u\in L^\infty(0,T;L^2(\om))$,
$\tau\in [0,T)$ and $\chi_{(\tau,T)}$ stands for the characteristic
function of $(\tau,T)$. In this equation, controls are restricted
over $\omega\times (\tau,T)$. It is well known that for each $u\in
L^\infty(0,T;L^2(\om))$ and each $y_0\in L^2(\om)$, Equation
(\ref{1.1}) has a unique solution in $C([0,T]; L^2(\Omega))$. {\it
We denote, by $y(\cdot;\chi_{(\tau,T)}u, y_0)$, the solution of
Equation (\ref{1.1}) corresponding to the control $u$ and the
initial state $y_0$.   } Throughout this paper, $\|\cdot\|$ and $<
\cdot, \cdot>$ stand for the usual norm and inner product of the
space $L^2(\Omega)$, respectively.

Next, we will set up, for each $y_0\in L^2(\om)$,  three kinds of
optimal control problems associated with Equation (\ref{1.1}).  For
this purpose, we take a target $z_d\in L^2(\om)$ such that
\begin{eqnarray}\label{1.2}
 z_d\notin \Bigr\{y(T; \chi_{(0,T)}u,0): ~u\in  L^\infty(0,T;L^2(\om))
 \Bigl\}.
\end{eqnarray}
The set on the right hand side of (\ref{1.2}) is called the
attainable set of  Equation (\ref{1.1}). Then we
introduce the following target sets:
$$
B(z_d,r)=\{\hat{y}\in L^2(\om) : \|\hat{y}-z_d\|\leq r\},\;\; r>0.
$$
For each $M\geq 0$, each $r>0$ and each $\tau\in [0,T)$, we define
three sets of controls as follows:

\begin{itemize}
\item
$ \mathcal{U}_{\tau,M}=\{v\in L^\infty(0,T;L^2(\om)): \|v(t)\|\leq
M\; \; \mbox{for a.e.}\;\; t\in (\tau,T)\} $;

\item $\mathcal{U}_{M,r}=\{ v : \exists \; \tau\in [0,T)\; \mbox{s.t.}\;
v\in \mathcal{U}_{\tau,M}\;\;\mbox{and}\;\; y(T;
\chi_{(\tau,T)}v,y_0)\in B(z_d,r)\}$;

\item
$ \mathcal{U}_{ r,\tau}=\{v\in L^\infty(0,T;L^2(\om)): y(T;
\chi_{(\tau,T)}v,y_0)\in B(z_d,r)\}. $

\end{itemize}
For each $u\in \mathcal{U}_{M,r}$, we set
\begin{eqnarray}\label{1.2-1}
\widetilde{\tau}_{M,r}(u)=\sup\{ \tau\in [0,T): u\in \mathcal{U}_{\tau,M}\; \mbox{and}\;   y(T;
\chi_{(\tau,T)}u,y_0)\in B(z_d,r)\}.
\end{eqnarray}

Three kinds of  optimal control problems studied in this paper are
as follows:

\begin{itemize}
\item
 \textbf{ $(OP)^{\tau,M}$:} $\inf \{\|y(T; \chi_{(\tau,T)}u,y_0)-z_d\|^2 : u\in
\mathcal{U}_{\tau,M}\}$;

\item
~\textbf{$(TP)^{M,r}$:} $\sup\{\widetilde{\tau}_{M,r}(u) : u\in
\mathcal{U}_{M,r}\}$;

\item \textbf{ $(NP)^{r,\tau}$:} $\inf\{
\|u\|_{L^\infty(\tau,T;L^2(\om))} : u\in \mathcal{U}_{ r, \tau}\}$.
 \end{itemize}
 We call
$(OP)^{\tau,M}$ as an  optimal target control problem, which is a
kind of  optimal control problem with the observation of the final
state  (see \cite{Lions}, page 177). The problem $(NP)^{r,\tau}$ is
an optimal norm control problem, which is related to the approximate
controllability (see \cite{FPZ}). The problem $(TP)^{M,r}$ is an
optimal time control problem. The aim of controls in $(TP)^{M,r}$ is to delay initiation of
active control  as late as possible, such that the corresponding
solution reaches the target $B(z_d,r)$ at the ending time $T$ (see
\cite{MS}).

 The above three  problems provide the
following three values, respectively:
\begin{itemize}

\item
$r(\tau,M)\equiv  \inf  \{\|y(T; \chi_{(\tau,T)}u,y_0)-z_d\| : u\in
\mathcal{U}_{\tau,M}\}; $
\item
$\tau(M,r)\equiv \sup\{\widetilde{\tau}_{M,r}(u) : u\in
\mathcal{U}_{M,r}\}$;
\item
$ M(r,\tau)\equiv \inf\{ \|u\|_{L^\infty(\tau,T;L^2(\om))} : u\in
\mathcal{U}_{ r, \tau }\}. $
\end{itemize}
The value $r(\tau,M)$ is called the optimal distance to the target for $(OP)^{\tau,M}$; while  values $\tau(M,r)$   and $M(r,\tau)$  are called the
optimal time for $(TP)^{M,r}$ and the optimal norm for
$(NP)^{r,\tau}$, respectively. The optimal controls to these
problems are defined as follows:

\begin{itemize}
\item $u^*$ is called  an  optimal control  to
$(OP)^{\tau,M}$  if $u^*=\chi_{(\tau,T)}v^*$ for some
$v^*\in \mathcal{U}_{\tau,M}$ such that
$\|y(T;\chi_{(\tau,T)}v^*,y_0)-z_d\|=r(\tau,M)$;

\item $u^*$ is called an optimal control to $(TP)^{M,r}$
if $u^*=\chi_{(\tau(M,r),T)}v^*$ for some  $v^*\in \mathcal{U}_{\tau(M,r),M}$ such that
$y(T; \chi_{(\tau(M,r),T)}v^*,y_0)\in B(z_d,r).$

\item $u^*$ is called  an
optimal control  to $(NP)^{r,\tau}$ if $u^*=\chi_{(\tau,T)}v^*$ for some $v^*\in \mathcal{U}_{r,\tau}$ and
$\|u^*\|_{L^\infty(\tau,T;L^2(\om))}=M(r,\tau)$.

\end{itemize}

 Throughout the  paper, the following notation will be used frequently:
\begin{eqnarray}\label{3.2}
r_T(y_0)\equiv\|y(T;0, y_0)-z_d\|.
\end{eqnarray}

The main purpose of this study is to present an equivalence
theorem for the above-mentioned three kinds of optimal control
problems and its applications. This theorem can be stated, in plain language, as follows:
\begin{itemize}
\item $(OP)^{\tau,M}\Leftrightarrow (TP)^{M,r(\tau,M)}
\Leftrightarrow (NP)^{r(\tau,M),\tau}$\;\;when $M>0$ and $\tau\in
[0,T)$;

\item  $(NP)^{r,\tau}\Leftrightarrow (OP)^{\tau,M(r,\tau)}\Leftrightarrow
(TP)^{M(r,\tau),r}$ when $r\in (0,r_T(y_0))$ and $\tau\in [0,T)$;

\item  $(TP)^{M,r}\Leftrightarrow (NP)^{r,\tau(M,r)}\Leftrightarrow  (OP)^{\tau(M,r),M}$ when $M>0$ and $r\in [r(0,M),r_T(y_0))$.
 \end{itemize}
Here, by $(P_1)\Leftrightarrow (P_2)$, we mean that  problems
$(P_1)$ and $(P_2)$ have the same optimal controls.  Based on the
equivalence theorem, the study  of one kind of optimal control
problem can be carried out by investigating one of the other two kinds of
optimal control problems. In particular, one can use some existing
fine properties for optimal target controls to derive properties of
optimal norm controls and optimal time controls.

An important application of the equivalence theorem is to build up a feedback law for
norm optimal control problems. We will roughly  present this result in what follows. Notice that
 Problem $(NP)^{r,\tau}$
depends on $\tau\in [0,T)$ and  $y_0\in L^2(\om)$, when $r$  and $z_d$ are fixed. To stress this dependence, we denote, by $(NP)^{r,\tau}_{y_0}$, the problem $(NP)^{r,\tau}$ with
the initial state $y_0$.
Throughout this paper, we let $A$ be the operator on $L^2(\Omega)$ with domain $D(A)= H^1_0(\om)\bigcap H^2(\om)$ and defined by $Ay=\triangle y$ for each $y\in D(A)$. Write
$\{e^{t\triangle}: \; t\geq 0\}$
for the semigroup generated by $A$.
 By the equivalence theorem and some characteristics of the
target optimal control problems, we construct a map $F:[0,T)\times L^2(\Omega)\rightarrow L^2(\om)$ holding   properties:

\noindent $(i)$ For each $y_0\in L^2(\Omega)$ and each $\tau\in [0,T)$, the evolution equation
\begin{eqnarray*}
\left\{\begin{array}{ll} \ns\dot{y}(t)-A
y(t)=\chi_\omega\chi_{(\tau,T)}(t)F(t,y(t)),\q &t\in (0,T)\\
\nm y(0)=y_0,
\end{array}\right.
\end{eqnarray*}
 has a unique mild solution, which will be  denoted by
 $y_{F,\tau,y_0}(\cdot)$. Here $\chi_\omega$ is treated as an operator on $L^2(\Omega)$ in the usual way.

\noindent ${(ii)}$ For each $y_0\in L^2(\om)$ and each $\tau\in [0,T)$, $ \chi_{(\tau,T)}(\cdot)F( \cdot,y_{F,\tau,y_0}(\cdot))$ is the optimal control to
Problem $(NP)^{r,\tau}_{y_0}$.

\noindent Consequently,  the map $F$ is an optimal feedback
law for the  family
of  optimal norm control problems as follows:
$
\big\{(NP)^{r,\tau}_{y_0} \; : \;\tau\in [0,T),\, y_0\in
L^2(\om)\big\}.
$

With the aid of  the equivalence theorem, we also build up two algorithms  for the optimal norm, along with the
optimal control, to $(NP)^{r,\tau}$  and  the optimal time, together with
the optimal control, to $(TP)^{M,r}$, respectively. These algorithms  show that the optimal norm
and the optimal control to $(NP)^{r,\tau}$ can be approximated through solving a series of  two-point boundary value problems, and the same can be said about the optimal time and the optimal control to $(TP)^{M,r}$.

It deserves to  mention that all  results obtained in this paper
still stand when  Equation (\ref{1.1}) is replaced by
 \bb\label{1.7}
\left\{\begin{array}{lll} \p_ty-\triangle y+ay=\chi_\omega
\chi_{(\tau,T)}u&\mbox{in}&\om\times
(0,T),\\
y=0&\mbox{on}&\p \om\times (0,T),\\
y(0)=y_0 &\mbox{in}&\om,
\end{array}\right.\ee
where $a\in L^\infty(\om\times (0,T))$ and $\om$ is convex (see
Remark~\ref{remark3.11}).

The equivalence between optimal time and  norm control problems have
been studied in \cite{WZ}, \cite{GL} and \cite{FAT} and the
references therein. The  optimal time control problem studied in
these papers is to initiate control from the beginning such that the
corresponding solution (to a controlled system) reaches a target set
in the shortest time.
 Though problems studied in the current paper  differ from
those in \cite{WZ},  our study is partially inspired  by \cite{WZ}.
To the best of our knowledge, the equivalence theorem of the
above-mentioned three kinds of optimal control problems has not been
touched upon. Moreover, the feedback law and the algorithms
established in this paper seem to be new.

The rest of the paper is organized as follows:  Section 2 presents
the equivalence theorem and its proof. Section 3 provides  the above-mentioned two algorithms. In section 4, we build up an optimal norm feedback law.

\section{Equivalence of three optimal control problems}
$\;\;\;\;$Throughout this section, the initial state  $y_0$ is fixed
in $L^2(\om)$. For simplicity, we  write $y(\cdot;\chi_{(\tau,T)}u)$
and $r_T$ for $y(\cdot;\chi_{(\tau,T)}u,y_0)$ and $r_T(y_0)$ (which is defined by (\ref{3.2})),
respectively. The purpose of this section is to prove the following
equivalence theorem:

\begin{Theorem}\label{theorem3.10}
When $M>0$ and $\tau\in [0,T)$, the problems $(OP)^{\tau,M}$,
$(TP)^{M,r(\tau,M)}$ and $(NP)^{r(\tau,M),\tau}$ have the same
optimal control; When $r\in (0,r_T)$ and $\tau\in [0,T)$, the
problems $(NP)^{r,\tau}$, $(OP)^{\tau,M(r,\tau)}$ and
$(TP)^{M(r,\tau),r}$ have the same optimal control; When $M>0$ and
$r\in [r(0,M),r_T)$, the problems $(TP)^{M,r}$, $(NP)^{r,\tau(M,r)}$
and $(OP)^{\tau(M,r),M}$ have the same optimal control.
\end{Theorem}

\subsection{Some properties on  optimal target control problems}

\begin{Lemma}\label{lemma2.1}
Let $M\geq 0$ and $\tau\in [0,T)$. Then,  $(i)$  $(OP)^{\tau,M}$ has optimal
controls; $(ii)$  $r(\tau,M)>0$;  $(iii)$ $u^*$ is  an optimal control
 to $(OP)^{\tau,M}$  if and only if  $u^*\in L^\infty(0,T;L^2(\Omega))$, with  $u^*=0$ over $(\tau,T)$,   satisfies
\begin{eqnarray}\label{2.1}
\int_0^T< \chi_{(\tau,T)}(t)\chi_\omega  p^*(t),
u^*(t)>dt=\max_{v(\cdot)\in
\mathcal{U}_{\tau,M}}\int_0^T<\chi_{(\tau,T)}(t)\chi_\omega p^*(t),
v(t)>dt,
\end{eqnarray}
where $p^*$ is the solution to the equation:
\bb\label{2.2} \left\{\begin{array}{lll} \p_tp^*+\triangle
p^*=0&\mbox{in}&\om\times
(0,T),\\
p^*=0&\mbox{on}&\p \om\times (0,T),\\
p^*(T)=-(y^*(T)-z_d)&\mbox{in}&\om
\end{array}\right.
\ee
with $y^*(\cdot)$ solving the equation:
 \bb\label{2.3} \left\{\begin{array}{lll} \p_ty^*-\triangle
y^*=\chi_\omega\chi_{(\tau,T)} u^*&\mbox{in}&\om\times
(0,T),\\
y^*=0&\mbox{on}&\p \om\times (0,T),\\
y^*(0)=y_0&\mbox{in}&\om.
\end{array}\right.
\ee
\end{Lemma}
\begin{proof} $(i)$ and $(iii)$ have been proved in \cite{Lions} (see the proof of Theorem
7.2, Chapter III in \cite{Lions}). The remainder is to  show $(ii)$. For this purpose, we let $u^*$ be an optimal control to $(OP)^{\tau,M}$ and write $y^*(\cdot)$ for $y(\cdot;\chi_{(\tau,T)}u^*,y_0)$.
Then it holds that $r(\tau, M)=\|y^*(T)-z_d\|$ and
\begin{eqnarray}\label{2.3-9.30}
y^*(T)\in \big\{y(T;\chi_{(0,T)}u,y_0) \; :\; u\in L^\infty(0,T;L^2(\om))\big\}.
\end{eqnarray}
On the other hand, by the null controllability for the heat equation (see, for instance, \cite{DZZ} or \cite{EZ}), one can easily check that
$$\big\{y(T;\chi_{(0,T)}u,y_0)\; :\; u\in L^\infty(0,T;L^2(\om))\big\}=\big\{y(T;\chi_{(0,T)}u,0)\; :\; u\in L^\infty(0,T;L^2(\om))\big\}.$$
This, along with  (\ref{2.3-9.30}) and the assumption (\ref{1.2}), indicates that
that
$y^*(T)\neq z_d,$
which implies that
 $r(\tau, M)>0.$ This completes the proof.
\end{proof}

\begin{Lemma}\label{lemma2.3}
Let $M\geq 0$ and $\tau\in [0,T)$. Then, $(i)$  $u^*$ is an
optimal control  to
$(OP)^{\tau,M}$ if and only if   $u^*\in L^\infty(0,T; L^2(\Omega))$, with $u^*=0$ over $(0,\tau)$,  satisfies the  following equality:
\begin{eqnarray}\label{2.6}
u^*(t)=M\frac{\chi_\omega p^*(t)}{\|\chi_\omega
p^*(t)\|},\;\;\mbox{for a.e.}\;\; t\in (\tau,T),
\end{eqnarray}
where $p^*$ is the solution to (\ref{2.2}), with $y^*(\cdot)$ solving the equation (\ref{2.3});
$(ii)$ $(OP)^{\tau,M}$ holds the bang-bang property: any
optimal control $u^*$ satisfies that $\|u^*(t)\|=M$
for a.e. $t\in (\tau,T)$; $(iii)$ the optimal control of
$(OP)^{\tau,M}$ is unique.
\end{Lemma}

\begin{proof} First, the maximal condition (\ref{2.1}) is equivalent to the
following condition:
\begin{eqnarray}\label{2.4}
<\chi_\omega p^*(t), u^*(t)>=\max_{v^0\in B(0,M)}<\chi_\omega
p^*(t), v^0>\;\;\mbox{for a.e.}\;\; t\in (\tau,T),
\end{eqnarray}
where $B(0,M)$ is the closed ball (in $L^2(\om)$), centered at the
origin and of radius $M$. Since $p^*(T)=-(y^*(T)-z_d)\neq 0$ (see
$(ii)$ of Lemma~\ref{lemma2.1}), it follows from the unique
continuation property of the heat equation (see \cite{FHL}) that
\begin{eqnarray}\label{2.5}
\|\chi_\omega p^*(t)\|\neq 0\;\;\mbox{for each}\;\; t\in [0,T).
\end{eqnarray}
Thus, the condition (\ref{2.4}) is equivalent to the condition
(\ref{2.6}). This, along with $(iii)$ of Lemma~\ref{lemma2.1}, yields $(i)$.  Next, $(ii)$  follows
at once from (\ref{2.6}). Finally, $(iii)$ follows  from $(ii)$ (see \cite{FAT} or
\cite{Wang}). This completes the proof.
\end{proof}

\begin{Lemma}\label{lemma2.6} Let $M\ge0$  and $\tau\in [0,T)$. Then
the  two-point boundary value problem: \bb\label{2.9}
\left\{\begin{array}{ccll} \p_t\varphi-\Delta \varphi=M
\chi_{(\tau,T)} \displaystyle\frac{\chi_\omega \psi}{\|\chi_\omega
\psi\|},\;\; &\p_t\psi+\triangle \psi=0&\mbox{in}&\om\times
(0,T),\\
\ns\varphi=0,\;\;& \psi=0&\mbox{on}&\p \om\times (0,T),\\
\ns\varphi(0)=y_0,\;\; &\psi(T)=-(\varphi(T)-z_d)&\mbox{in}&\om
\end{array}\right.
\ee admits a unique solution $(\varphi^{\tau,M}, \psi^{\tau,M})$ in
$C([0,T]; L^2(\om))\times C([0,T]; L^2(\om))$.
 Furthermore, the control, defined by
\begin{eqnarray}\label{2.9-1}
u^{\tau,M}(t)= M\chi_{(\tau,T)}(t)\frac{\chi_\omega
\psi^{\tau,M}(t)}{\|\chi_\omega \psi^{\tau,M}(t)\|}, \;\; t\in [0,
T),
\end{eqnarray}
is the  optimal control to $(OP)^{\tau,M}$, while $\varphi^{\tau,M}$
is the corresponding optimal state. Consequently, it holds that
\begin{equation}\label{2.10}\|\varphi^{\tau,M}(T)-z_d\|=r(\tau,M).\end{equation}
\end{Lemma}

\begin{proof}
By Lemma~\ref{lemma2.1}, $(OP)^{\tau,M}$ has an optimal control
 $u^*$.   Let $y^*$ and $p^*$ be the corresponding solutions to Equation (\ref{2.3}) and Equation (\ref{2.2}), respectively. Clearly, they belong to
$C([0,T]; L^2(\om))$. It follows from $(i)$ of  Lemma~\ref{lemma2.3} that  $u^*$ satisfies
(\ref{2.6}). This, together with (\ref{2.2}) and (\ref{2.3}), shows
that  $(y^*, p^*)$ solves  Equation (\ref{2.9}).

Next, we prove the uniqueness. Suppose that $(\varphi_1,\psi_1)$ and
$(\varphi_2,\psi_2)$ are two solutions of Equation (\ref{2.9}).
Define $u_1$ and $u_2$ by (\ref{2.6}), where $p^*$ is replaced by
$\psi_1$ and $\psi_2$, respectively. It follows from  $(i)$ of Lemma~\ref{lemma2.3} that  $u_1$ and $u_2$ are the optimal control to $(OP)^{\tau,M}$ and
 $\varphi_i(\cdot)=y(\cdot; \chi_{(\tau,T)}u_i)$, $i=1,2$. Then by $(iii)$ of Lemma~\ref{lemma2.3}, $\varphi_1=\varphi_2$. Thus,
 it holds that
$\psi_1(T)=\psi_2(T)$, from which, it follows that $\psi_1=\psi_2$.

Finally, if $(\varphi^{\tau,M}, \psi^{\tau,M})$ is the solution of
Equation (\ref{2.9}),  then it follows from  $(i)$ of Lemma~\ref{lemma2.3}
 that $u^{\tau,M}$ (defined by (\ref{2.9-1})) and
$\varphi^{\tau,M}$ are the optimal control and  the optimal state to
$(OP)^{\tau,M}$. This completes the proof.
\end{proof}

\begin{Remark}\label{remark2.5} {\rm The unique continuation
property (\ref{2.5}) for the adjoint equation plays a very important
role in this paper. This property also holds for the adjoint
equation of Equation (\ref{1.7}),   where $\om$ is convex (see
\cite{PW}). With the help of this fact, one can easily check that
all results in previous lemmas still stand when the controlled
system is Equation (\ref{1.7})}.
\end{Remark}

\subsection{Equivalence of optimal target and norm  control problems}

\begin{Lemma}\label{lemma3.1}
Let $\tau\in [0,T)$. Then the map $M\rightarrow r(\tau,M)$ is
strictly monotonically decreasing and Lipschitz continuous from
$[0,\infty)$ onto $(0,r_T]$. Furthermore, it holds that
\begin{eqnarray}\label{3.11}
r=r(\tau,M(r,\tau))\;\;\mbox{for each}\;\; r\in (0,r_T]
\end{eqnarray} and
\begin{eqnarray}\label{3.12}
M=M(r(\tau,M), \tau)\;\; \mbox{for each}\;\; M\geq 0.
\end{eqnarray}
Consequently, for each $\tau\in [0,T)$, the maps $M\rightarrow
r(\tau,M)$  and $r\rightarrow M(r,\tau)$  are the inverse of each
other.

\end{Lemma}
\begin{proof}
 The
proof will be carried out by several steps as follows:

\noindent{\it Step 1.  It holds that $r(\tau,0 )=r_T$ and
$\lim_{M\rightarrow\infty}r(\tau,M)=0$.}

 The first equality above follows directly from  the definitions of $r_T$ and $r(\tau,0 )$. Now, we prove the second one. Let
$\varepsilon>0$. By the approximate controllability for the heat
equation (see \cite{FPZ}),  there is a control $u_\varepsilon\in
L^\infty(\tau,T; L^2(\om))$  such that
$y(T;\chi_{(\tau,T)}u_\varepsilon)\in B(z_d,\varepsilon)$.
Clearly,  $u_\varepsilon\in
\mathcal{U}_{\tau,M}$ for all $M\geq
\|u_\varepsilon\|_{L^\infty(\tau,T; L^2(\om))}$.
 Then, by
the optimality of $r(\tau,M)$, we deduce that
$
r(\tau,M)\leq\|y(T;\chi_{(\tau,T)}u_\varepsilon)-z_d\|\leq
\varepsilon$ for each $ M\geq
\|u_\varepsilon\|_{L^\infty(\tau,T; L^2(\om))},
$
 from which, it follows that $\lim_{M\rightarrow\infty}r(\tau,M)=0$.

\noindent{\it Step 2.  The map $M\rightarrow r(\tau,M)$ is strictly
monotonically decreasing.}

Let $0\leq M_1<M_2$. We claim that $r(\tau,M_2)<r(\tau,M_1)$.
Seeking for a contradiction,  suppose that $r(\tau,M_2)\geq
r(\tau,M_1)$. Then  optimal control $u_1$ to $(OP)^{\tau,M_1}$ would
satisfy that
$
\|y(T; \chi_{(\tau,T)}u_1)-z_d\|=r(\tau,M_1)\leq
r(\tau,M_2)$ and $u_1\in\mathcal{U}_{\tau,M_1}\subset
\mathcal{U}_{\tau,M_2}.
$
These yield that $u_1$ is the optimal control to  $(OP)^{\tau,M_2}$.
By the bang-bang property of  $(OP)^{\tau,M_2}$ (see $(ii)$ of
Lemma~\ref{lemma2.3}), it holds that $\|u_1(t)\|=M_2$ for a.e. $t\in
(\tau,T)$. This contradicts to  that $u_1\in
\mathcal{U}_{\tau,M_1}$, since $M_1<M_2$.

\noindent{\it Step 3. The map $M\rightarrow r(\tau,M)$ is Lipschitz
continuous.}

 Let $M_1, M_2\in [0,\infty)$. Without loss of generality, we can assume that
$0\leq M_1< M_2$.  Let $u^*$ be optimal control to $(OP)^{\tau,M_2}$.
Then by
the monotonicity of the map $M\rightarrow r(\tau,M)$ and  the optimality
of  $u^*$ to $(OP)^{\tau,M_2}$, we see that
$$
\ba{ll}
\ns &\ds r(\tau,M_1)> r(\tau,M_2)=\left\| e^{T\triangle}
y_0+\int^T_{\tau}e^{(T-s)\triangle}u^*(s)ds-z_d\right\|\\
\ns\ge&\ds\left\| e^{T\triangle}
y_0+\int^T_{\tau}e^{(T-s)\triangle}\frac{M_1}{M_2}u^*(s)ds-z_d\right\|-
\frac{(M_2-M_1)}{M_2}\left\|\int^T_{\tau}e^{(T-s)\triangle}u^*(s)ds\right\|.\ea
$$
Since $\displaystyle\frac{M_1}{M_2}u^*\in \mathcal{U}_{\tau,M_1}$, it follows from the definition of
$r(\tau,M_1)$ that
$$
\ds\left\| e^{T\triangle}
y_0+\int^T_{\tau}e^{(T-s)\triangle}\frac{M_1}{M_2}u^*(s)ds-z_d\right\|\geq r(\tau,M_1).
$$
Because $\|u^*\|_{L^\infty(\tau,T;L^2(\om))}\leq M_2$, we find that
$$
\int_\tau^T\|e^{(T-s)\triangle}\|\|u^*(s)\|ds\leq M_2(T-\tau).
$$
Putting the above three estimates together leads to the estimate as follows:
$$
r(\tau,M_1)>r(\tau,M_2)\geq r(\tau,M_1)-(M_2-M_1),
$$
from which, it follows that
$$| r(\tau,M_1)- r(\tau,M_2) |\le |M_1-M_2|(T-\tau)\;\;\mbox{for all}\;\;
M_1,M_2\in [0,\infty).
$$

\noindent{\it Step 4. The proof of (\ref{3.11})}

First of all, by the definitions of $r_T$ , one can easily check
that $M( r_T,\tau)=0$ and
\begin{eqnarray}\label{3.13}
r_T=r( \tau,0)=r(\tau,M(r_T,\tau)).
\end{eqnarray}

Then, let $r\in (0,r_T)$. By Step 2,  $M(r,\tau)>0$
for this case. We are going to prove the following two claims:

\noindent{\it Claim one: $r\geq r(\tau,M(r,\tau))$} and {\it Claim two:  $r\leq r(\tau,M(r,\tau))$.}

\noindent Clearly, these claims, together with (\ref{3.13}), lead to (\ref{3.11}).
To prove the first claim, we let $u$ be an optimal control to
$(NP)^{r,\tau}$ (the existence of  such a control is provided in
\cite{FPZ}). Then it holds that
$
\|y(T; \chi_{(\tau,T)}u)-z_d\|\leq r$ and $ u\in
\mathcal{U}_{\tau,M(r,\tau)}$.
These,  along with the definition of $r(\tau,M)$,  shows {\it Claim
one}.

Now we show the second claim. Seeking a contradiction, suppose
that $r>r(\tau,M(r,\tau))$. Since the map $M\rightarrow r(\tau,M)$
is continuous and strictly monotonically decreasing,  there would be
a $M_1\in (0,M(r,\tau))$ such that $r(\tau,M_1)=r$. Thus, the
optimal control $u_1$  to $(OP)^{\tau,M_1}$ satisfies that
\begin{eqnarray}\label{3.15}
\|u_1\|_{L^\infty(\tau,T;L^2(\om))}=M_1<M(r,\tau)\;\;\mbox{and}\;\;\|y(T;\chi_{(\tau,T)}u_1)-z_d\|=r(\tau,M_1)=r.
\end{eqnarray}
The second equality in (\ref{3.15}) implies that $u_1\in
\mathcal{U}_{r,\tau}$, which, together with  the optimality of
$M(r,\tau)$, indicates  that $M(r,\tau)\leq
\|u_1\|_{L^\infty(\tau,T;L^2(\om))}$. This  contradicts to the first
inequality in  (\ref{3.15}).

\noindent{\it Step 5. The proof of (\ref{3.12})}.

One can easily check that $r(\tau,M)\in (0,r_T]$ whenever $M\geq 0$
and $\tau\in [0,T)$. Thus, we can make use of (\ref{3.11}) to get
that
\begin{eqnarray}\label{3.16}
r(\tau,M)=r(\tau,M(r(\tau,M),\tau)),\;\;M\geq 0, \tau\in [0,T).
\end{eqnarray}
Since the map $M\rightarrow r(\tau,M)$ is strictly monotone,
(\ref{3.12}) follows from (\ref{3.16}) at once.

In summary, we complete the proof.
\end{proof}

\begin{Proposition}\label{theorem3.3}
$(i)$ The optimal control to $(OP)^{\tau,M}$, where $M\geq 0$ and $\tau\in
[0,T)$,
 is an optimal control to
$(NP)^{r(\tau,M),\tau}$. $(ii)$ Any optimal control to
$(NP)^{r,\tau}$, where $\tau\in [0,T)$ and $r\in (0,r_T]$, is the
optimal control to $(OP)^{\tau,M(r,\tau)}$. $(iii)$ For each  $\tau\in [0,T)$ and  each $r\in (0,r_T]$,
$(NP)^{r,\tau}$  holds the
bang-bang property (i.e.,  any optimal control $u^*$ satisfies that $\|u^*(t)\|=M(r,\tau)$ for a.e. $t\in (\tau,T)$) and the optimal control to $(NP)^{r,\tau}$ is unique.
\end{Proposition}
\begin{proof}
$(i)$ The optimal control $u$ to  $(OP)^{\tau,M}$  satisfies that
$
y(T;\chi_{(\tau,T)}u)\in
B(z_d,r(\tau,M))$, $\|u\|_{L^\infty(\tau,T;L^2(\om))}=M$ and $u=0$ over $(0,\tau)$.
These, together with (\ref{3.12}), indicate that
 $u$ is an optimal control
to $(NP)^{r(\tau,M),\tau}$. $(ii)$ An optimal control $v$ to
$(NP)^{r,\tau}$, where $\tau\in [0,T)$ and $r\in (0,r_T]$, satisfies that
$
\|v\|_{L^\infty(\tau,T;L^2(\om))}=M(r,\tau)$, $\|y(T;\chi_{(\tau,T)}v)-z_d\|
\leq r$ and $v=0$ over $(0,\tau)$.
These, along with (\ref{3.11}), yields that
 that $v$ is the optimal
control to $(OP)^{\tau,M(r,\tau)}$.
$(iii)$ The bang-bang property and the uniqueness of $(NP)^{r,\tau}$  follow from $(ii)$ and Lemma~\ref{lemma2.3}.
 This completes the proof.

\end{proof}

\subsection{Equivalence of optimal norm and time  control problems}

\begin{Lemma}\label{lemma3.5}
Let  $r\in (0,r_T)$ and $M\geq M(r,0)$. Then,  $(TP)^{M,r}$ has
optimal  controls. Moreover, it holds that $\tau(M,r)<T$.
\end{Lemma}
\begin{proof}
We first claim that when $u\in\mathcal{U}_{M,r}$, the  supremum in
(\ref{1.2-1}) can be reached, i.e.
\begin{eqnarray}\label{3.18-0}
y(T;\chi_{(\widetilde{\tau}(u),T)}u)\in B(z_d,r)\; \mbox{and}\; u\in \mathcal{U}_{\widetilde{\tau}(u) ,M} \;\; \mbox{when
}\;\;u\in\mathcal{U}_{M,r}.
\end{eqnarray}
Here, we simply write $\widetilde{\tau}(u)$ for
$\widetilde{\tau}_{M,r}(u)$, which is defined by (\ref{1.2-1}). To
this end, we let $u\in\mathcal{U}_{M,r}$.  Then by the definition of
$\widetilde{\tau}(u)$, there is a sequence $\{\tau_n\}\subset [0,T)$
such that $\tau_n\rightarrow\widetilde{\tau}(u)$,
$y(T;\chi_{(\tau_n,T)}u)\in B(z_d,r)$ and $u\in\mathcal{U}_{\tau_n,M}$. From these, (\ref{3.18-0})
follows at once.

Next we notice that  $(NP)^{r,0}$ has optimal controls (see \cite{FPZ}) and  any optimal control to $(NP)^{r,0}$ belongs to $\mathcal{U}_{M(r,0),r}\subset \mathcal{U}_{M,r}$ (since $M\geq
M(r,0)$). These imply  that $\mathcal{U}_{M,r}\neq\emptyset$.
Thus, there is a sequence $\{u_n\}\subset \mathcal{U}_{M,r}$ such
that $\widetilde{\tau}(u_n)\rightarrow \tau(M,r)$.
On the other hand,  by (\ref{3.18-0}),
$y(T;\chi_{(\widetilde{\tau}(u_n),T)}u_n)\in B(z_d,r)$ and $u_n\in\mathcal{U}_{\widetilde{\tau}(u_n),M}$.
Hence, there exist a subsequence of $\{u_n\}$, still denoted in the same way, and a control $v^*\in L^\infty(0,T;L^2(\om))$ such that
 $$\chi_{(\widetilde{\tau}(u_n),T)}u_n\rightarrow
\chi_{(\tau(M,r),T)}v^*\;\;\mbox{weakly star in}\; L^\infty(0,T;L^2(\om))$$ and $$y(T;\chi_{(\widetilde{\tau}(u_n),T)}u_n)\rightarrow y(T;\chi_{(\tau(M,r),T)}v^*).$$
From these, it follows that
$y(T;\chi_{(\tau(M,r),T)}v^*)\in B(z_d,r)$ and $v^*\in \mathcal{U}_{\tau(M,r),M}$.
Hence, $\chi_{(\tau(M,r),T)}v^*$ is an optimal control to $(TP)^{M,r}$.

Finally, since $r<r_T$ and $y(T;\chi_{(\tau(M,r),T)}v^*)\in B(z_d,r)$, it follows that $\tau(M,r)<T$.
This completes the proof.
\end{proof}

\begin{Lemma}\label{lemma3.6}
Let $r\in (0,r_T)$. Then the map $\tau\rightarrow M(r,\tau)$ is
strictly monotonically increasing and continuous  from $[0,T)$ onto
$[M(0,\tau), \infty)$. Furthermore, it holds that
\begin{eqnarray}\label{3.27}
M=M(r, \tau(M,r))\;\;\mbox{for each}\;\; M\in [M(r,0), \infty)
\end{eqnarray}
and
\begin{eqnarray}\label{3.28}
\tau=\tau(M(r,\tau),r)\;\;\mbox{for each}\;\; \tau\in [0,T).
\end{eqnarray}
Consequently, the maps $\tau\rightarrow M(r,\tau)$  and
$M\rightarrow \tau(M,r)$  are the inverse  of each other.

\end{Lemma}

\begin{proof}
We carry out the proof by several steps as follows:

\noindent{\it Step 1. This map is strictly monotonically increasing over
$[0,T)$.}

Let $0\leq\tau_1<\tau_2<T$. We claim that $M(r,\tau_1)<M(r,\tau_2)$.
Seeking for a contradiction, suppose that $M(r,\tau_2)\leq
M(r,\tau_1)$. Then the optimal control $u_2$ to $(NP)^{r,\tau_2}$
 would
satisfy
\begin{eqnarray*}
\|\chi_{(\tau_2,T)}u_2\|_{L^\infty(0,T;L^2(\Omega))}=M(r,\tau_2)\leq
M(r,\tau_1)\;\;\mbox{and}\;\; y(T;\chi_{(\tau_2,T)}u_2)\in B(z_d,r).
\end{eqnarray*}
These imply that $\chi_{(\tau_2,T)}u_2$ is the optimal control to
$(NP)^{r,\tau_1}$. Then,  it follows from the bang-bang property of
$(NP)^{r,\tau_1}$ (see Proposition~\ref{theorem3.3}) that
$
\|\chi_{(\tau_2,T)}u_2(t)\|=M(r,\tau_1)$ over $(\tau_1,\tau_2)$.
This contradicts to the facts that  $\tau_1<\tau_2$ and
$M(r,\tau_1)>0$ (which follows from $r<r_T$).

\noindent{\it Step 2.  $0\leq\tau_1<\tau_2<\cdots<\tau_n\rightarrow
\tau<T\Rightarrow M(r,\tau_n)\rightarrow M(r,\tau)$. }

If this did not hold, then by the monotonicity of the map
$\tau\rightarrow M(r,\tau)$,
 we would have
\begin{eqnarray}\label{3.19}
\lim_{n\rightarrow\infty}M(r,\tau_n)=M(r,\tau)-\delta\;\;\mbox{for
some}\;\; \delta>0.
\end{eqnarray}
Let $u_n$ and  $y_n$  be the optimal control and the optimal state
to $(OP)^{\tau_n,M(r,\tau_n)}$, respectively. Then, it follows from
Lemma~\ref{lemma2.1} that
\begin{eqnarray}\label{3.20}
\int_{0}^T<\chi_\omega p_n, \chi_{(\tau_n,T)} u_n>dt\geq
\int_{0}^T<\chi_\omega p_n, \chi_{(\tau_n,T)}v_n>dt\;\;\mbox{for
each}\;\; v_n\in \mathcal{U}_{\tau_n,M(r,\tau_n)},
\end{eqnarray}
\begin{eqnarray*}
\left\{
\begin{array}{ll}
\p_ty_n-\triangle y_n=\chi_\omega \chi_{(\tau_n,T)} u_n & \textrm{in } \Omega\times (0,T),\\
y_n=0 & \textrm{on } \partial\Omega\times(0,T),\\
y_n(0)=y_{0} & \textrm{in }\Omega,
\end{array}
\right.
\end{eqnarray*}
\begin{eqnarray*}
\left\{
\begin{array}{ll}
\p_tp_n+\triangle p_n=0 & \textrm{in } \Omega\times (0,T),\\
p_n=0 & \textrm{on } \partial\Omega\times(0,T),\\
p_n(T)=-(y_n(T)-z_d) & \textrm{in }\Omega.
\end{array}
\right.
\end{eqnarray*}
Besides, by the optimality of $y_n$  and (\ref{3.11})
(in Lemma~\ref{lemma3.1}), we see that
\begin{eqnarray}\label{3.21}
\|y_n(T)-z_d\|=r(M(r,\tau_n), \tau_n)=r\;\;\mbox{for all}\;\; n\in
\mathbb{N}.
\end{eqnarray}
Since $\tau_n\rightarrow\tau$
 and $\|u_n\|_{L^\infty(\tau_n,T;L^2(\Omega))}= M(r,\tau_n)\leq
M(r,\tau)-\delta$, there exist  a subsequence, still denoted in the
same way, and a control $\widetilde{u}\in L^\infty(0, T; L^2(\Omega))$ such that
\begin{eqnarray}\label{3.22}
\chi_{(\tau_n,T)}u_n\rightarrow
\chi_{(\tau,T)}\widetilde{u}\;\;\mbox{weakly star in}\;\;
L^\infty(0, T; L^2(\Omega)).
\end{eqnarray}
This, together with the equations satisfied by $y_n$ and $p_n$
respectively, indicates that
\begin{eqnarray}\label{3.23}
y_n\rightarrow \widetilde{y}\;\;\mbox{and}\;\; p_n\rightarrow
\widetilde{p}\;\;\mbox{in}\;\; C([0,T];L^2(\Omega)),
\end{eqnarray}
\begin{eqnarray}\label{3.24}
\left\{
\begin{array}{ll}
\p_t\widetilde{y}-\triangle \widetilde{y}=\chi_\omega \chi_{(\tau,T)} \widetilde{u} & \textrm{in } \Omega\times (0,T),\\
\widetilde{y}=0 & \textrm{on } \partial\Omega\times(0,T),\\
\widetilde{y}(0)=y_{0} & \textrm{in }\Omega
\end{array}
\right.
\end{eqnarray}
and
\begin{eqnarray}\label{3.25}
\left\{
\begin{array}{ll}
\p_t\widetilde{p}+\triangle \widetilde{p}=0 & \textrm{in } \Omega\times (0,T),\\
\widetilde{p}=0 & \textrm{on } \partial\Omega\times(0,T),\\
\widetilde{p}(T)=-(\widetilde{y}(T)-z_d) & \textrm{in }\Omega.
\end{array}
\right.
\end{eqnarray}
In addition,  it follows from  (\ref{3.21}) and (\ref{3.23})
that $\|\widetilde{y}(T)-z_d\|=r$. By making use of (\ref{3.11})
again, we deduce that
\begin{eqnarray}\label{3.26}
\|\widetilde{y}(T)-z_d\|=r(\tau,M(r,\tau)).
\end{eqnarray}

Now, we take a $v\in \mathcal{U}_{\tau,M(r,\tau)-\delta}$. Since
$M(r,\tau)-\delta>0$, it holds that
$$
\frac{M(r,\tau_n)}{M(r,\tau)-\delta}\chi_{(\tau_n,T)}v\in
\mathcal{U}_{\tau_n,M(r,\tau_n)}.
$$
Then, it follows from (\ref{3.20}) that
\begin{eqnarray*}
\int_{0}^T<\chi_\omega p_n, \chi_{(\tau_n,T)}u_n)>dt\geq
\int_{0}^T<\chi_\omega p_n,
\frac{M(r,\tau_n)}{M(r,\tau)-\delta}\chi_{(\tau_n,T)}v>dt.
\end{eqnarray*}
By (\ref{3.19}), (\ref{3.22}) and (\ref{3.23}), we can pass to the
limit in the above to get that
\begin{eqnarray*}
\int_{0}^T<\chi_\omega \widetilde{p}\;,
\chi_{(\tau,T)}\widetilde{u}>dt\geq \int_{0}^T<\chi_\omega
\widetilde{p}\; , \chi_{(\tau,T)}v >dt\;\;\mbox{for all}\;\; v\in
\mathcal{U}_{\tau, M(r,\tau)-\delta, }.
\end{eqnarray*}
This, along with the fact that $\widetilde{u}\in\mathcal{U}_{\tau,
M(r,\tau)-\delta, }$ (which follows from (\ref{3.22})), indicates
that
\begin{eqnarray*}
\int_{0}^T<\chi_\omega \widetilde{p}\; ,
\chi_{(\tau,T)}\widetilde{u}>dt =\displaystyle{\max_{v\in
\mathcal{U}_{\tau, M(r,\tau)-\delta, }}} \int_{0}^T<\chi_\omega
\widetilde{p}\; , \chi_{(\tau,T)}>dt.
\end{eqnarray*}
According to Lemma~\ref{lemma2.1}, the above equality, together with
 (\ref{3.24}) and (\ref{3.25}), shows that $\chi_{(\tau,T)}\widetilde{u}$ and
$\widetilde{y}$ are the optimal control and the optimal state to
$(OP)^{\tau, M(r,\tau)-\delta, }$.
 Therefore, it
stands that $ \|\widetilde{y}(T)-z_d\|=r(\tau, M(r,\tau)-\delta ),
$ which, combined with (\ref{3.26}), indicates that $
r(\tau,M(r,\tau) )=r(\tau, M(r,\tau)-\delta). $ This contradicts
with the strict monotonicity of the map $M\rightarrow r(\tau,M)$
(see Lemma~\ref{lemma3.1}).

\noindent{\it Step 3.   $T>\tau_1>\cdots>\tau_n\rightarrow \tau\geq
0\Rightarrow M(r,\tau_n)\rightarrow M(r,\tau)$. }

If this did not hold, then by the monotonicity of the map
$\tau\rightarrow M(r,\tau)$, we would have that
$
 \lim_{n\rightarrow\infty}M(r,\tau_n)=M(r,\tau)+\delta \;\;\mbox{for some}\;\; \delta>0.
 $
 Following the same argument as that in Step 2, we can derive that
 $
r(\tau,M(r,\tau))=r(\tau, M(r,\tau)+\delta). $ This contradicts to
the strict monotonicity of the map $M\rightarrow r(\tau,M)$.

\noindent{\it Step 4. $ \lim_{\tau\rightarrow T}M(r,\tau)=\infty. $}

Seeking for a contradiction, we suppose that
$
0<\tau_1<\cdots<\tau_n\rightarrow T$ and $
\lim_{n\rightarrow\infty} M(r,\tau_n)=M<\infty$. Let $u_n$ and
$y_n$ be the optimal control and state for $(NP)^{r,\tau_n}.$ Then
we would have that
$\chi_{(\tau_n,T)}u_n\rightarrow
0$ weakly star in $L^\infty(0,T;L^2(\Omega))$ and $y_n(\cdot)\rightarrow y(\cdot;0)$ in
$C([0,T];L^2(\Omega))$.
 Thus, it holds that $r_T\equiv\|y(T;0)-z_d\|=\lim_{n\rightarrow\infty}\|y_n(T)-z_d\|\leq
 r$, which contradicts to the assumption that $r<r_T$.

\noindent{\it Step 5. The proof of (\ref{3.27})}

By Lemma~\ref{lemma3.5}, the problem $(TP)^{M,r}$ has an optimal
control  $u$. It holds that
\begin{eqnarray}\label{3.29}
y(T; \chi_{(\tau(M,r),T)}u)\in B(z_d,r)\;\;\mbox{ and}\;\;
\|u\|_{L^\infty(\tau(M,r),T;L^2(\om))}\leq M.
\end{eqnarray}
From the first fact in (\ref{3.29}), we see that $u\in
\mathcal{U}_{r,\tau(M,r)}$. This, together with the optimality of
$M(r,\tau)$ and the second fact in (\ref{3.29}), shows that
\begin{eqnarray}\label{3.30}
M\geq M(r,\tau(M,r)).
\end{eqnarray}

 Seeking for a contradiction,  suppose that $M>
M(r,\tau(M,r))$. Since the map $\tau\rightarrow M(r,\tau)$ is
continuous and strictly monotonically increasing,  there would be a
 $\tau_1$, with $\tau_1\in (\tau(M,r), T)$, such that
$M(r,\tau_1)=M$. Clearly, the optimal control  $u_1$  to
$(NP)^{r,\tau_1}$ satisfies that
\begin{eqnarray}\label{3.31}
\|u_1\|_{L^\infty(\tau_1,T;
L^2(\om))}=M(r,\tau_1)=M\;\;\mbox{and}\;\; y(T;
\chi_{(\tau_1,T)}u_1)\in B(z_d,r).
\end{eqnarray}
From these, it follows that $u_1\in \mathcal{U}_{M,r}$. Then, by the
optimality of $\tau(M,r)$ , (\ref{1.2-1}) and (\ref{3.31}), we
deduce that $ \tau(M,r)\geq \widetilde{\tau}(u_1)\geq \tau_1, $
which contradicts with  that $\tau_1\in (\tau(M,r), T)$.

\noindent{\it Step 6. The proof of (\ref{3.28})}.

Let $\tau\in [0,T)$. By Step 1, it follows that $M(r,\tau)\geq
M(r,0)$. Then we can apply (\ref{3.27}) to deduce that
$M(r,\tau)=M(r,\tau(M(r,\tau),r))$. By making use of Step 1 again,
we obtain that $\tau=\tau(M(r,\tau),r)$.

In summary, we  complete the proof.
\end{proof}

\begin{Proposition}\label{theorem3.8} $(i)$
 Any optimal control to $(TP)^{M,r}$, where $r\in (0,r_T)$ and
 $M\geq M(r,0)$,  is the   optimal control to $(NP)^{r, \tau(M,r)}$.  $(ii)$ The optimal optimal control to
  $(NP)^{r, \tau}$, with
$\tau\in [0,T)$ and $r\in (0,r_T)$, is an optimal control to
$(TP)^{M(r,\tau),r}$. $(iii)$ For each   $r\in
(0,r_T)$ and each
 $M\geq M(r,0)$, $(TP)^{M,r}$ holds the bang-bang property (i.e., any optimal control  $u^*$ satisfies that $\|u^*(t)\|=M$ for a.e.
 $t\in (\tau(M,r), T)$)  and the optimal control to $(TP)^{M,r}$ is unique.

\end{Proposition}

\begin{proof}
$(i)$ An optimal control  $u$ to $(TP)^{M,r}$ satisfies that $u=0$ over $(\tau(M,r),T)$,
$$y(T; \chi_{(\tau(M,r),T)}u)\in B(z_d,r)\;\mbox{ and }\;
\|u\|_{L^\infty(\tau(M,r),T; L^2(\om))}\leq M.$$
These, together
with (\ref{3.27}), yields that  $u$ is the optimal control to
$(NP)^{r, \tau(M,r)}$. $(ii)$ The optimal control  $v$  to
$(NP)^{r,\tau}$ satisfies that $v=0$ over $(\tau,T)$,
$\|u\|_{L^\infty(\tau,T;L^2(\om))}=M(r,\tau)$ and $y(T;
\chi_{(\tau,T)}u)\in B(z_d,r)$. These, together with (\ref{3.28}),
yields that
 $u$ is an optimal control to $(TP)^{M(r,\tau),r}$.
 $(iii)$ The bang-bang property  and the
 uniqueness of  $(TP)^{M,r}$ follow from $(iii)$ of  Proposition~\ref{theorem3.3} and $(i)$ above.
 This
completes the proof.
\end{proof}

\subsection{Equivalence of optimal target and time  control problems}

$\;\;\;\;$Though the equivalence between optimal target and time
control problems can be derived from Proposition~\ref{theorem3.3}
and
     Proposition~\ref{theorem3.8}, the properties of maps
     $\tau\rightarrow r(\tau,M)$ and $r\rightarrow \tau(M,r)$ are
     independently interesting and will be used in the next section.
     This is why  we introduce what follows.
\begin{Lemma}\label{lemma2.17}
Let $M>0$. Then the map $\tau\rightarrow r(\tau,M)$ is strictly
monotonically increasing and continuous from $[0,T)$ onto $[r(0,M),
r_T)$. Furthermore, it holds that
\begin{eqnarray}\label{3.41}
r=r(\tau(M,r),M)\;\;\mbox{for each}\;\; r\in [r(0,M), r_T),
\end{eqnarray}
\begin{eqnarray}\label{3.42}
\tau=\tau(M,r(\tau,M))\;\;\mbox{for each}\;\; \tau\in [0,T).
\end{eqnarray}
Consequently, the maps $\tau\rightarrow r(\tau,M)$   and
$r\rightarrow\tau(M,r)$  are the inverse of each other.
\end{Lemma}

\begin{proof}
We carry out the proof by several steps as follows:

\noindent{\it Step 1.  The map $\tau\rightarrow r(\tau,M)$  is strictly
monotonically increasing. }

Let $0\leq \tau_1<\tau_2<T$. It follows from (\ref{3.12}) that
\begin{eqnarray}\label{3.42-1}
M(r(\tau_1,M),\tau_1)=M(r(\tau_2,M),\tau_2).
\end{eqnarray}
We first  claim that
\begin{eqnarray}\label{3.42-1-1}
r(\tau_2,M)\in (0, r_T)\;\; \mbox{when}\;\; M>0.
\end{eqnarray}
In fact, on one hand, it is clear that $r(\tau_2,M)>0$ (see
Lemma~\ref{lemma2.1}). On the other hand, since the map
$M\rightarrow r(\tau_2,M)$ is strictly monotonically decreasing (see
Lemma~\ref{lemma3.1}), it holds that
$r(\tau_2,M)<r(\tau_2,0)=\|y(T;0)-z_d\|=r_T$. Then by
(\ref{3.42-1-1}),  we can apply Lemma~\ref{lemma3.6} to get that
$M(r(\tau_2,M),\tau_2)>M(r(\tau_2,M),\tau_1)$. This, together with
(\ref{3.42-1}), yields that
\begin{eqnarray}\label{3.24-2}
M(r(\tau_1,M),\tau_1)>M(r(\tau_2,M),\tau_1).
\end{eqnarray}
Since the map $r\rightarrow M(r,\tau_1)$ is strictly monotonically
decreasing (see Lemma~\ref{lemma3.1}), it follows from
(\ref{3.24-2}) that $r(\tau_1,M)<r(\tau_2,M)$.

\noindent{\it Step 2.  The map $\tau\rightarrow r(\tau,M)$   is continuous.}

Since for each $\tau\in [0,T)$, the map $r\rightarrow M(r,\tau)$ is
continuous and monotonic over $(0,r_T)$ (see Lemma~\ref{lemma3.1}),
and for each $r\in (0,r_T)$, the map $\tau\rightarrow M(r,\tau)$ is
continuous (and monotonic) over $[0,T)$ (see Lemma~\ref{lemma3.6}),
it follows that
\begin{eqnarray}\label{3.24-3}
\mbox{the map}\;\; (r,\tau)\rightarrow M(r,\tau)\;\; \mbox{is
continuous over}\;\; (0,r_T)\times [0,T).
\end{eqnarray}
Now we prove that the map $\tau\rightarrow r(\tau,M)$ is continuous
from left. For this purpose, we let $0\leq
\tau_1<\tau_2<\cdots<\tau_n\rightarrow \tau<T$. Then by the
monotonicity of $\{\tau_n\}$, $\lim_{n\rightarrow\infty}r(\tau_n,M)$
exists. Thus, it follows from (\ref{3.24-3}) that
$$
\lim_{n\rightarrow\infty}M(r(\tau_n,M),\tau_n)=M(\lim_{n\rightarrow\infty}r(\tau_n,M),
\tau).
$$
On the other hand, by (\ref{3.12}), it stands that
$$
M(r(\tau_n,M),\tau_n)=M=M(r(\tau,M),\tau)\;\;\mbox{for all}\;\;n.
$$
These yield that $M(\lim_{n\rightarrow\infty}r(\tau_n,M),
\tau)=M(r(\tau,M),\tau)$. This, together with the strict
monotonicity of the map $r\rightarrow M(r,\tau)$ (see
Lemma~\ref{lemma3.1}), indicates that
$\lim_{n\rightarrow\infty}r(\tau_n,M)=r(\tau,M)$. Thus, the map
$\tau\rightarrow r(\tau,M)$ is continuous from left. Similarly, we
can prove that it is continuous from right.

\noindent{\it Step 3. It holds that $\lim_{\tau\rightarrow T}r(\tau,M)=r_T$.}

Clearly, the optimal control  $u_\tau$ to $(OP)^{\tau,M}$ satisfies
 that
$\|y(T; \chi_{(\tau,T)}u_\tau)-z_d\|=r(\tau,M)$ and
$\|u_\tau\|_{L^\infty(\tau,T;L^2(\om))}\leq M.$ One can easily see
that $\chi_{(\tau,T)}u_\tau\rightarrow 0\;\;\mbox{in}\;\;
L^\infty(0,T;L^2(\om))$ as $\tau$ tends to $T$, from which, it
follows that $y(T;\chi_{(0,T)}u_\tau)\rightarrow y(T;0)$  as $\tau$
tends to $T$. Therefore, it holds that
$
r_T\equiv\|y(T;0)-z_d\|=\lim_{\tau\rightarrow
T}\|y(T;\chi_{(0,T)}u_\tau)-z_d\|=\lim_{\tau\rightarrow T}r(\tau,M).
$

\noindent{\it Step 4. The proof of (\ref{3.41}) and (\ref{3.42}).}

We start with proving the following:
\begin{eqnarray}\label{3.43}
\mathcal{A}_1=\mathcal{A}_2,
\end{eqnarray}
where $\mathcal{A}_1=\{ (M,r) : r\in (0,r_T), M\geq M(r,0)\}$ and
$\mathcal{A}_2=\{(M,r) : M>0, r\in[r(0,M),r_T)\}.
$
In fact, if  $(M,r)\in \mathcal{A}_1$,  since  $r>0$, it follows
that $M>0$. On the other hand, because  $M\geq M(r,0)$,  we can
apply Lemma~\ref{lemma3.1} to get that $ r(0,M)\geq r(0,M(r,0))=r. $
Thus, it stands that $(M,r)\in \mathcal{A}_2$.  Similarly, we can
prove that $\mathcal{A}_2\subset\mathcal{A}_1$.

Next, it follows from  (\ref{3.43}) and (\ref{3.27}) that $
M=M(r,\tau(M,r))$ when $ M>0$ and $r\in [r(0,M), r_T)$. This,
together with (\ref{3.11}), indicates that $$
r(\tau(M,r),M)=r(\tau(M,r),M(r,\tau(M,r)))=r\;\;\mbox{for each}\; r\in [r(0,M), r_T), $$ which leads to
(\ref{3.41}).

Finally, because $r(\tau,M)\in (0, r_T)$ (see (\ref{3.42-1-1})), we
can make use of (\ref{3.28}) to get that $ \tau(M(r(\tau,M),\tau),
r(\tau,M))=\tau, $ which, along with (\ref{3.12}), gives
(\ref{3.42}).

In summary, we  complete the proof.
\end{proof}

\begin{Proposition}\label{proposition2.19} The optimal control to  $(TP)^{M,r}$, where
$M>0$ and $r\in [r(0,M),r_T)$,   is  the optimal control to
$(OP)^{\tau(M,r),M}$. Conversely, the optimal control to
$(OP)^{\tau,M}$, where $M>0$ and $\tau\in [0,T)$, is the optimal
control to $(TP)^{M,r(\tau,M)}$.
\end{Proposition}

This proposition can be directly derived from Lemma~\ref{lemma2.17}.
Also it is a consequence of Proposition~\ref{theorem3.3},
Proposition~\ref{theorem3.8} and (\ref{3.43}). We omit its proof.

\subsection{Proof of Theorem~\ref{theorem3.10}}

$\;\;\;\;$Let $(P_1)$ and $(P_2)$ be two optimal control problems.
By $(P_1)\Rightarrow (P_2)$, we mean that the optimal control to
$(P_1)$ is the optimal control to  $(P_2)$. The proof will be
carried out by several steps as follows:

\noindent{\it Step 1.  $(OP)^{\tau,M}\Rightarrow (TP)^{M,r(\tau,M)}
\Rightarrow (NP)^{r(\tau,M),\tau}\Rightarrow(OP)^{\tau,M}$, $M>0$,
$\tau\in [0,T)$. }\\

\noindent $(OP)^{\tau,M}\Rightarrow (TP)^{M,r(\tau,M)}$: It follows from
Proposition~\ref{theorem3.3}.\\

\noindent $(TP)^{M,r(\tau,M)}
\Rightarrow (NP)^{r(\tau,M),\tau}$: We first claim that
\begin{eqnarray}\label{2.5.1}
r(\tau,M)\in (0,r_T)\;\; \mbox{when}\;\; M>0\;\;\mbox{and}\;\;
\tau\in [0,T).
\end{eqnarray}
In fact, it follows from  Lemma~\ref{lemma2.1} that $r(\tau,M)>0$.
On the other hand, since $M>0$ and the map $M\rightarrow r(\tau,M)$
is strictly monotonically decreasing (see Lemma~\ref{lemma3.1}), it
holds that $r(\tau,M)<r(0,\tau)=r_T$. These lead to (\ref{2.5.1}).

We next claim that
\begin{eqnarray}\label{2.5.2}
M\geq M(r(\tau,M),0)\;\; \mbox{when}\;\; M>0\;\;\mbox{and}\;\;
\tau\in [0,T).
\end{eqnarray}
Indeed, since the map $\tau\rightarrow r(\tau,M)$ is monotonically
increasing (see Lemma~\ref{lemma3.6}), it holds that $r(0,M)\leq
r(\tau,M)$. Because  the map  $r\rightarrow M(r,0)$ is monotonically
decreasing (see  Lemma~\ref{lemma3.1}), it stands that
$M(r(0,M),0)\geq M(r(\tau,M),0)$. This, combined with (\ref{3.12}),
shows (\ref{2.5.2}). Now, by (\ref{2.5.1}) and (\ref{2.5.2}), we can
apply Proposition~\ref{theorem3.8}, together with (\ref{3.28}), to
get $(TP)^{M,r(\tau,M)} \Rightarrow (NP)^{r(\tau,M),\tau}$.\\

\noindent $(NP)^{r(\tau,M),\tau}\Rightarrow(OP)^{\tau,M}$: By (\ref{2.5.1}), we can make use of
Proposition~\ref{theorem3.3}, together with (\ref{3.12}), to get
$(NP)^{r(\tau,M),\tau}\Rightarrow(OP)^{\tau,M}$.\\

\noindent{\it Step 2.  $(NP)^{r,\tau}\Rightarrow
(OP)^{\tau,M(r,\tau)}\Rightarrow (TP)^{M(r,\tau),r}\Rightarrow
(NP)^{r,\tau}$, $r\in (0,r_T)$, $\tau\in [0,T)$.  }\\

\noindent$(NP)^{r,\tau}\Rightarrow (OP)^{\tau,M(r,\tau)}$: It  follows from
Proposition~\ref{theorem3.3}.\\

\noindent $(OP)^{\tau,M(r,\tau)}\Rightarrow (TP)^{M(r,\tau),r}$: We
first claim that
\begin{eqnarray}\label{2.5.3}
M(r,\tau)>0\;\; \mbox{when}\;\; r\in (0,r_T)\;\;\mbox{and}\;\;
\tau\in [0,T).
\end{eqnarray}
In fact, since $r_T=\|y(T;0)-z_d\|$, it holds that $M(r_T,\tau)=0$.
On the other hand, since $r<r_T$ and the map $r\rightarrow
M(r,\tau)$ is strictly monotonically decreasing (see
Lemma~\ref{lemma3.1}), we see that $M(r,\tau)>M(r_T,\tau)$. Thus,
(\ref{2.5.3}) follows immediately. Now, by (\ref{2.5.3}), we can
apply Proposition~\ref{proposition2.19}, along with (\ref{3.11}), to
derive $(OP)^{\tau,M(r,\tau)}\Rightarrow (TP)^{M(r,\tau),r}$.\\

\noindent$(TP)^{M(r,\tau),r}\Rightarrow (NP)^{r,\tau}$: Since $r\in (0,r_T)$, the map $\tau\rightarrow M(r,\tau)$
is monotonically increasing (see Lemma~\ref{lemma3.6}). Thus, it
holds that $M(r,\tau)\geq M(r,0)$. Then we can make use of
Proposition~\ref{theorem3.8}, together with (\ref{3.28}), to yield
$(TP)^{M(r,\tau),r}\Rightarrow (NP)^{r,\tau}$.\\

\noindent{\it Step 3. $(TP)^{M,r}\Rightarrow (NP)^{r,\tau(M,r)}\Rightarrow
(OP)^{\tau(M,r),M}\Rightarrow (TP)^{M,r}$, $M>0$, $r\in
[r(0,M),r_T)$.}\\

\noindent  $(TP)^{M,r}\Rightarrow (NP)^{r,\tau(M,r)}$: It follows from (\ref{3.43}) and Proposition~\ref{theorem3.8}. \\

\noindent $(NP)^{r,\tau(M,r)}\Rightarrow (OP)^{\tau(M,r),M}$:
Since  $r>0$ in this case (see (\ref{3.43})), we can apply
Proposition~\ref{theorem3.3}, together with (\ref{3.12}), to get
$(NP)^{r,\tau(M,r)}\Rightarrow (OP)^{\tau(M,r),M}$.\\

\noindent $(OP)^{\tau(M,r),M}\Rightarrow
(TP)^{M,r}$: It follows from  Proposition~\ref{proposition2.19}, together with
(\ref{3.41}).\\

In summary, we  complete the proof of Theorem~\ref{theorem3.10}.
\endpf

\begin{Remark}\label{remark3.11}
{\rm  All results in this section hold for the case where the
controlled system is Equation (\ref{1.7}). In fact, these results
hold for the three kinds of optimal control problems studied in this
paper, when the adjoint equation of the controlled heat equation has
the unique continuation property (\ref{2.5}).

}
\end{Remark}

\section{Applications I: Algorithms for $M(r,\tau)$ and $\tau(M,r)$  }

$\;\;\;\;$ Throughout this section, we fix an initial state $y_0\in
L^2(\om)$ and write $r_T$ for  $r_T(y_0)$. For each $M>0$ and
$\tau\in [0,T)$, $(\varphi^{\tau,M},\psi^{\tau,M})$ denotes the
unique solution to the two-point boundary value problem (\ref{2.9})
and $\varphi^{\tau,M}$ (or $\psi^{\tau,M})$) stands for the first (or second) component of this
solution when it appears alone.

\begin{Proposition}\label{theorem3.1}
Let $\tau\in [0,T)$ and $r\in (0,r_T)$. Then $M^*$, $u^*$ and $y^*$
are the optimal norm, the optimal control and the optimal state to
$(NP)^{r,\tau}$ if and only if $M^*$, $u^*$ and $y^*$ satisfy that $M^*>0$,
\begin{eqnarray}\label{a3.1}
\|y^*(T)-z_d\|=r,
\end{eqnarray}
\begin{eqnarray}\label{a3.2}
u^*(t)=M^*\chi_{(\tau,T)}(t)\displaystyle\frac{\chi_\omega\psi^{\tau,
M^*}(t)}{\|\chi_\omega\psi^{\tau, M^*}(t)\|},\;\;t\in [\tau,T)
\end{eqnarray}
and
\begin{eqnarray}\label{a3.3}
y^*(t)=\varphi^{\tau, M^*}(t),\;\; t\in [0,T].
\end{eqnarray}
\end{Proposition}

\begin{proof}
Suppose that  $M^*$, $u^*$ and $y^*$ are the optimal norm, the
optimal control and the optimal state to $(NP)^{r,\tau}$. Clearly,   $M^*=M(r,\tau)$.
It follows from Lemma~\ref{lemma3.1} that $M(r,\tau)>M(r_T,\tau)$. Hence,  $M^*>0$.
 Then, by Theorem~\ref{theorem3.10}, $u^*$ and
$y^*$ are the optimal control and the optimal state to
$(OP)^{\tau,M(r,\tau)}= (OP)^{\tau, M^*}$, respectively. On the other
hand, it follows from  Lemma~\ref{lemma2.6} that
$M^*\chi_{(\tau,T)}\displaystyle\frac{\chi_\omega\psi^{\tau,
M^*}}{\|\chi_\omega\psi^{\tau, M^*}\|}$ and $y^{\tau, M^*}$ are also
the optimal control and the optimal state to $(OP)^{\tau, M^*}$.
Then, by the uniqueness of the optimal control to this problem,
(\ref{a3.2}) and (\ref{a3.3}) follow at once. Besides, by the
optimality of $y^*$ to $(OP)^{\tau,M(r,\tau)}$, we see that
$\|y^*(T)-z_d\|=r(\tau,M(r,\tau))$. This, together with
(\ref{3.11}),
 gives (\ref{a3.1}).

Conversely, suppose that  a triplet $(M^*, u^*,y^*)$, with $M^*>0$, enjoys
(\ref{a3.1}), (\ref{a3.2}) and (\ref{a3.3}). According to
Lemma~\ref{lemma2.6}, it follows from (\ref{a3.2}) and (\ref{a3.3})
that $u^*$ and $y^*$ are  the optimal control and the optimal state
to $(OP)^{\tau, M^*}$ and that  $\|y^*(T)-z_d\|=r(M^*,\tau)$, which,
together with (\ref{a3.1}), shows that $r=r(M^*,\tau)$. Then, by
Theorem~\ref{theorem3.10}, $u^*$ and $y^*$ are  the optimal control
and the optimal state to $(NP)^{r(M^*,\tau),\tau}=(NP)^{r,\tau}$.
Hence,
$\|u^*\|_{L^\infty(\tau,T;L^2(\om))}=M(r,\tau)$, which, along with
(\ref{a3.2}), indicates that $M^*=M(r,\tau)$, i.e., $M^*$ is the
optimal norm to $(NP)^{r,\tau}$.  This completes the proof.
\end{proof}

By Theorem~\ref{theorem3.10}, Lemma~\ref{lemma2.6} and
Lemma~\ref{lemma2.17}, following a very similar argument used to prove
Proposition~\ref{theorem3.1}, we can verify the following property for
$(TP)^{M,r}$.

\begin{Proposition}\label{prop3.2}
Let $r\in (0,r_T)$ and $M\geq M(r,0)$. Then $\tau^*$, $u^*$ and
$y^*$ are the optimal time, the optimal control and the optimal
state to $(TP)^{M,r}$ if and only if $\tau^*$, $u^*$ and $y^*$ satisfy that $\tau^*\in [0,T)$,
\begin{eqnarray*}
\|y^*(T)-z_d\|=r,
\end{eqnarray*}
\begin{eqnarray*}
u^*(t)=M\chi_{(\tau^*,T)}(t)\displaystyle\frac{\chi_\omega\psi^{\tau^*,M}(t)}{\|\chi_\omega\psi^{\tau^*,
M}(t) \|},\;\;t\in (\tau^*,T)
\end{eqnarray*}
and
\begin{eqnarray*}
y^*(t)=\varphi^{\tau^*, M}(t),\;\; t\in [0,T].
\end{eqnarray*}
\end{Proposition}

The above two propositions not only are independently interesting, but
also hint us to find two algorithms for the  optimal norm, together
with the optimal  control, to $(NP)^{r,\tau}$ and the optimal time,
along with the optimal  control, to $(TP)^{M,r}$, respectively.
First of all, we build up, corresponding to each $r\in (0,r_T)$ and
each $\tau\in (0,T)$, a sequence of numbers as follows:

\begin{itemize}
\item {\bf Structure of $\{M_n\}_{n=0}^\infty$:}  Let $M_0>0$ be arbitrarily taken.
Let $K\in \mathbb{N}$ be such that
 $$
K=\min\{  k : r( \tau,kM_0)<r, k=1,2,\cdots \}.
$$
( The existence of such a $K$ is guaranteed by Lemma~\ref{lemma3.1}.)
Set $a_0=0$ and $b_0= KM_0$.  Write
$M_1=\displaystyle\frac{a_0+b_0}{2}$.  In general, when
$M_n=\displaystyle\frac{a_{n-1}+b_{n-1}}{2}$ with $a_{n-1}$ and
$b_{n-1}$ being given,
 it is defined that
\begin{eqnarray*}
\{a_{n},b_{n}\}=
 \left\{
\begin{array}{ll}
\{M_{n}, b_{n-1}\}   & \;\mbox{if}\;\; r(\tau,M_{n})>r,\\
\{a_{n-1}, M_{n}\}   & \;\mbox{if}\;\; r(\tau,M_{n})\leq r
\end{array}
 \right.
\end{eqnarray*}
 and
$M_{n+1}=\displaystyle\frac{a_n+b_n}{2}$.

\end{itemize}

\begin{Remark}\label{remark3.3}{\rm
Let $\tau\in [0,T)$ and $r\in (0,r_T)$ be given. For each $M\geq 0$, we can
determine the value   $r(\tau,M)$ by solving  the two-point
boundary value problem (\ref{2.9}) corresponding to $(\tau, M)$, since  $r(\tau,M)=\|\varphi^{
\tau,M}(T)-z_d\|$ (see  Lemma~\ref{lemma2.6}).
Clearly,  $M_1$ is
determined by $K$. Since the map $M\rightarrow r(\tau,M)$ is strictly  monotonically decreasing and $r(\tau,M)$ tends to $0$ as $M$ goes to $\infty$ (see Lemma~\ref{lemma3.1}),
  $K$ can be determined by solving limited number of two-point
boundary value problems (\ref{2.9}) corresponding to $(\tau,M)$ with $M= kM_0$,
$k=1,2,\cdots,K$. On the other hand, when $n\geq 1$
$M_{n+1}$ is determined by $\varphi^{M_n,\tau}$, which can be solved
from (\ref{2.9}) corresponding to  $(\tau,M_n)$. In summary, we conclude that the
sequence $\{M_n\}_{n=0}^\infty$ can be
 solved from a series of two-point boundary value problems (\ref{2.9}) corresponding to  $(\tau, M)$, with
$M= kM_0$, $k=1,2,\cdots,K$ and with $M=M_n$, $n=1,2,\cdots$.}

\end{Remark}

\begin{Theorem}\label{proposition3.2}
Suppose that  $r\in(0,r_T)$ and $\tau\in [0,T)$. Let
$\{M_n\}_{n=0}^\infty$ be the sequence built up above. Let
$u_n=M_n\chi_{(\tau,T)}\displaystyle\frac{\chi_\omega\psi^{\tau,M_n}}{\|\chi_\omega
\psi^{\tau,M_n}\|}$ and $u^*$ be the optimal control  to
$(NP)^{r,\tau}$. Then it holds that
\begin{eqnarray}\label{a3.4}
M_n\rightarrow M(r,\tau)
\end{eqnarray}
and
\begin{eqnarray}\label{a3.6}
u_n\rightarrow u^* \;\;\mbox{in}\;\;L^2(\tau,T;L^2(\om))\;\;
\mbox{and in}\;\;
 C([\tau,T-\delta]; L^2(\om))\;\;\mbox{for each}\;\;\delta\in
 (0,T-\tau).
\end{eqnarray}
\end{Theorem}

\begin{proof} For simplicity, we write $(\varphi_n,\psi_n)$ for the
solution $(\varphi^{\tau,M_n},\psi^{\tau,M_n})$ with $n=1,2,\cdots$.
We start with proving (\ref{a3.4}). From the structure of $\{M_n\}$,
it follows  that
 $M_{n}\in [a_{n}, b_{n}]\subset [a_{n-1},b_{n-1}]$
and $b_n-a_n=\displaystyle\frac{b_{n-1}-a_{n-1}}{2}$. Thus, it holds
that
$\lim_{n\rightarrow\infty}a_n=\lim_{n\rightarrow\infty}b_n=\lim_{n\rightarrow\infty}M_n$.
Since the map $M\rightarrow r(\tau,M)$ is continuous (see
Lemma~\ref{lemma3.1}) and   $r(\tau,a_n)>r\geq r(\tau,b_n)$ (which
follows also from the structure of $\{M_n\}$), we find that
$r(\tau,\lim_{n\rightarrow\infty}M_n)=r$. This, along with
(\ref{3.11}), indicates that
\begin{eqnarray}\label{a3.7}
r(\tau,\lim_{n\rightarrow\infty}M_n)=r(\tau,M(r,\tau)).
\end{eqnarray}
 Then, (\ref{a3.4}) follows from (\ref{a3.7}) and  the strict
monotonicity of the map $M\rightarrow r(\tau,M)$ (see
Lemma~\ref{lemma3.1}).

Next, write  $y^*(\cdot)$ and $y_n(\cdot)$ for the solutions
$y(\cdot;\chi_{(\tau,T)}u^*)$ and  $y(\cdot; \chi_{(\tau,T)}u_n)$,
respectively.
  We claim that
\begin{eqnarray}\label{a3.7-1}
u_n\rightarrow u^*\;\;\mbox{weakly star in}\;\;
L^\infty(\tau,T;L^2(\om))\;\; \mbox{and}\;\; y_n\rightarrow
y^*\;\;\mbox{in}\;\; C([0,T];L^2(\om)).
\end{eqnarray}
 In fact, by the
definitions of $u_n$ and $y_n$, it follows from Lemma~\ref{lemma2.6}
that they are the optimal control and the optimal state to
$(OP)^{\tau,M_n}$, respectively. We arbitrarily take subsequences of
$\{u_n\}$ and $\{y_n\}$,  denoted by $\{u_{n_k}'\}$
  and   $\{y_{n_k}'\}$, respectively.  Clearly, there are subsequences $\{u_{n_k}\}$ of $\{u_{n_k}'\}$
and  $\{y_{n_k}\}$ of  $\{y_{n_k}'\}$ such that
\begin{eqnarray}\label{a3.8}
u_{n_k}\rightarrow \widetilde{u}\;\;\mbox{weakly star in }\;
L^\infty(\tau,T;L^2(\om))\;\;\mbox{and}\;\; y_{n_k}\rightarrow
\widetilde{y}\;\; \mbox{in}\;\;C([0,T];L^2(\om)),
\end{eqnarray}
where $\widetilde{y}(\cdot)=y(\cdot; \chi_{(\tau,T)}\widetilde{u})$.
 These, along with (\ref{a3.4}) and
 (\ref{a3.7}), indicate that
$$
\|\widetilde{u}\|_{L^\infty(\tau,T;L^2(\om))}\leq
\liminf_{k\rightarrow\infty}\|u_{n_k}\|_{L^\infty(\tau,T;L^2(\om))}=
\liminf_{k\rightarrow\infty}M_{n_k}=M(r,\tau)
$$
and
$$
\|\widetilde{y}(T)-z_d\|=\lim_{k\rightarrow\infty}\|y_{n_k}(T)-z_d\|=\lim_{n\rightarrow\infty}r(\tau,M_{n_k})=r(\tau,\lim_{k\rightarrow\infty}M_{n_k})=r(\tau,M(r,\tau)).
$$
From these, we see that  $\widetilde{u}$ and $\widetilde{y}$ are the
optimal control and the optimal state to $(OP)^{\tau,M(r,\tau)}$.
Then, according to Theorem~\ref{theorem3.10}, they are the optimal control and the optimal state to $(NP)^{r,\tau}$.
Since the optimal control to $(NP)^{r,\tau}$ is unique,
(\ref{a3.7-1}) follows from (\ref{a3.8}).

Now we verify the first convergence in (\ref{a3.6}). By  the first
convergence in (\ref{a3.7-1}), we see  that
\begin{eqnarray}\label{a3.7-2}
u_n\rightarrow u^*\;\;\mbox{weakly in}\;\; L^2(\tau,T;L^2(\om)).
\end{eqnarray}
On the other hand, according to  Proposition~\ref{theorem3.1}, it stands
that
\begin{eqnarray}\label{a3.9-1}
u^*(t)=M(r,\tau)\chi_{(\tau,T)}(t)\frac{\chi_\omega\psi^{\tau,M(r,\tau)}(t)}{\|\chi_\omega\psi^{\tau,M(r,\tau)}(t)\|},\;\;t\in
[0,T)
\end{eqnarray}
\begin{eqnarray}\label{a3.9-2}
y^*=\varphi^{\tau,M(r,\tau)}\;\; \mbox{and}\;\; \|y^*(T)-z_d\|=r.
\end{eqnarray}
By the definition of $u_n$, (\ref{a3.9-1}) and (\ref{a3.4}), we see
that
$$
\|u_n\|_{L^2(\tau,T;L^2(\om))}\rightarrow\|u^*\|_{L^2(\tau,T;L^2(\om))}.
$$
This, along with (\ref{a3.7-2}), yields the first convergence in
(\ref{a3.6}).

Finally, we  show the second convergence in (\ref{a3.6}). By  the
first equality of (\ref{a3.9-2}) and  the second convergence in
(\ref{a3.8}), we see that $y_n(T)\rightarrow
\varphi^{\tau,M(r,\tau)}(T)$ strongly in $L^2(\om)$.  This, together
with the equations satisfied by $\psi_n$ and
$\psi^{\tau,M(r,\tau)}$, respectively,  indicates that
\begin{eqnarray}\label{a3.10}
\psi_n\rightarrow \psi^{\tau,M(r,\tau)}\;\; \mbox{in}\;\;
C([0,T];L^2(\om)).
\end{eqnarray}
Then we arbitrarily fix a $\delta\in (0, T-\tau)$. By (\ref{a3.9-1})
and by the definition of $u_n$, after some simple computation, we
deduce that   for each $t\in [0,T-\delta]$,
\begin{eqnarray}\label{a3.11}
\begin{array}{ll}
\|u_n(t)-u^*(t)\|&\leq
|M_n-M(r,\tau)|+\displaystyle\frac{2M_n}{\|\chi_\omega\psi^{\tau,M(r,\tau)}(t)\|}\|\chi_\omega\psi_n(t)-\chi_\omega\psi^{\tau,M(r,\tau)}(t)\|.
\end{array}
\end{eqnarray}

On the other hand, by the second equality of (\ref{a3.9-2}) and
 the unique continuation property (see \cite{FHL}), it follows   that
$\|\chi_\omega\psi^{\tau,M(r,\tau)}(t)\|\neq 0$ for all $t\in
[0,T)$. This, together with  the continuity of
$\psi^{\tau,M(r,\tau)}(\cdot)$ over $[0,T-\delta]$, yields that
\begin{eqnarray}\label{a3.12}
\max_{t\in
[0,T-\delta]}\displaystyle\frac{1}{\|\chi_\omega\psi^{\tau,M(r,\tau)}(t)\|}\leq
C_\delta\;\; \mbox{for some positive}\;\; C_\delta.
\end{eqnarray}
Now, the second convergence in (\ref{a3.6}) follows immediately from
(\ref{a3.11}), (\ref{a3.4}),  (\ref{a3.10}),  and (\ref{a3.12}).
This completes the proof.

\end{proof}

We end this section by introducing  an algorithm for the optimal
time and the optimal control to $(TP)^{M,r}$. For each pair $(M,r)$
with $r\in (0,r_T)$ and $M\geq M(r,0)$,  we construct a sequence
$\{\tau_n\}_{n=0}^\infty\subset [0,T)$ as follows.

\begin{itemize}
\item {\bf Structure of $\{\tau_n\}_{n=1}^\infty$:}
Let $a_0=0$ and $b_0=T$. Set
$\tau_1=\displaystyle\frac{a_0+b_0}{2}$. In general, when
$\tau_{n}=\displaystyle\frac{a_{n-1}+b_{n-1}}{2}$ with $a_{n-1}$ and
$b_{n-1}$ being given, it is  defined that
\begin{eqnarray*}
\{a_{n},b_{n}\}=
 \left\{
\begin{array}{ll}
\{a_{n-1},\tau_n\}   & \;\mbox{if}\;\; r(\tau_n,M)>r,\\
\{\tau_n, b_{n-1}\}   & \;\mbox{if}\;\; r(\tau_n,M)\leq r
\end{array}
 \right.
\end{eqnarray*}
and  $\tau_{n+1}=\displaystyle\frac{a_{n}+b_{n}}{2}$.

\end{itemize}

\begin{Remark}\label{3.6} {\rm Since $r(\tau_n,M)=\|\varphi^{\tau_n,M}(T)-z_d\|$,  $\tau_{n+1}$ is
determined by $\varphi^{\tau_n,M}$, which can be solved from
(\ref{2.9}) corresponding to $\tau=\tau_n$.}
\end{Remark}

By Theorem~\ref{theorem3.10}, Lemma~\ref{lemma2.6},
Lemma~\ref{lemma2.17} and Proposition~\ref{prop3.2}, following  a very similar argument to prove
Theorem~\ref{proposition3.2}, we can verify the next
approximation result.

\begin{Theorem}\label{proposition3.5}
Suppose that $r\in (0,r_T)$ and $M\geq M(r,0)$. Let
$\{\tau_n\}_{n=1}^\infty$ be the sequence built up above. Let
$u_n=M\chi_{(\tau_n,T)}\displaystyle\frac{\chi_\omega\psi^{\tau_n,M}}{\|\chi_\omega\psi^{\tau_n,M}\|}$
and $u^*$ be the optimal control to $(TP)^{M,r}$. Then it holds that
\begin{eqnarray*}
\tau_n\rightarrow \tau(M,r)\;\; \mbox{as}\;\; n\rightarrow\infty
\end{eqnarray*}
and
\begin{eqnarray*}
u_n\rightarrow u^* \;\;\mbox{in}\;\;
L^2(\tau(M,r),T;L^2(\om))\;\;\mbox{and in}\;\;
 C([\tau(M,r),T-\delta]; L^2(\om))
\end{eqnarray*}
for each $\delta\in (0,T-\tau(M,r))$.
\end{Theorem}

\begin{Remark}\label{remark3.3} {\rm
$(i)$ From  the above-mentioned two algorithms,  we observe that the
optimal norm and the optimal control to $(NP)^{r,\tau}$ and the
optimal time and the optimal control to $(TP)^{M,r}$  can be
numerically solved, through numerically solving the two-point
boundary value problems (\ref{2.9}) with parameters $M$ and $\tau$
suitably chosen. \\$(ii)$ All results obtained in this section hold
for the case where the controlled system is Equation (\ref{1.7})
(see Remark~\ref{remark3.11}).}
\end{Remark}

\section{ Application II: Optimal Normal Feedback Law} $\;\;\;\;$\noindent

Throughout  this section,  we arbitrarily fix a  $r>0$.  We aim to
build up a feedback law for norm optimal control problems.

\subsection{Main results}

$\;\;\;\;$ We first introduce the following controlled equation:

\begin{equation}\label{state}
\left\{\ba{cl}
\ns \p_ty-\triangle y =\chi_{\omega} u ~\qq~&{\rm in}~\Omega\times(t_0,T),\\
\ns y=0 \qq~&{\rm on}~\partial\Omega\times(t_0,T),\\
\ns y(t_0)=y_0,\qq~&{\rm in}~\Omega\times(t_0,T).
\ea\right.
 \end{equation}
where  $(t_0,y_0)\in [0,T)\times L^2(\om)$.  Denote
by $y(\cdot;u,t_0,y_0)$ the solution to Equation (\ref{state})
corresponding to  the control $u$ and the initial data $(t_0,y_0)$.
Then, we define the following optimal target control and optimal
norm control problems.

\begin{itemize}
\item
 \textbf{ $(OP)^M_{t_0,y_0}$:} $\inf \{\|y(T; u,t_0,y_0)-z_d\|^2 : u\in L^\infty(t_0,T;B(0,M))\}$;

\item \textbf{ $(NP)_{t_0,y_0}$:} $\inf\{
\|u\|_{L^\infty(t_0,T;L^2(\om))} : u\in L^\infty(t_0,T;L^2(\om)),
y(T;u,t_0,y_0)\in B(z_d,r)\}$.
 \end{itemize}
 Throughout
this section,
\begin{itemize}
\item  $\bar u^M_{t_0,y_0}$ stands for the optimal control to
$(OP)^M_{t_0,y_0}$;
\item $N(t_0,y_0)$ denotes the optimal norm to
$(NP)_{t_0,y_0}$.
 \end{itemize}
 Thus $N(\cdot,\cdot)$ defines an optimal norm
functional over $[0,T)\times L^2(\om)$.

The only difference between optimal target control problems
$(OP)^{0,M}$ (which was introduced in Section 1) and
$(OP)^M_{t_0,y_0}$ is that the initial data for the first one is
$(0,y_0)$ while the initial data for the second one is $(t_0,y_0)$. The same can
be said about the norm optimal control problems. Therefore,
corresponding to each result about  $(OP)^{0,M}$  or  $(NP)^{r,0}$,
obtained in Section 2 or Section 3, there is an analogous version for $(OP)^M_{t_0,y_0}$
or  $(NP)_{t_0,y_0}$.

A  feedback law for the norm optimal control problems will be
established, with the aid of the equivalence between norm and target
optimal controls and some properties of $(OP)^M_{t_0,y_0}$. Those
properties are related to the following two-point boundary value problem
associated with $M\geq 0$, $t_0\in [0,T)$ and $y_0\in L^2(\om)$:
\bb\label{tpbv} \left\{\begin{array}{ccll} \p_ty-\Delta y=M
\displaystyle\frac{\chi_\omega \psi}{\|\chi_\omega \psi\|},\;\;
&\p_t\psi+\triangle \psi=0&\mbox{in}&\om\times
(t_0,T),\\
\ns y=0,\;\;& \psi=0&\mbox{on}&\p \om\times (t_0,T),\\
\ns y(t_0)=y_0,\;\; &\psi(T)=-(y(T)-z_d)&\mbox{in}&\om.
\end{array}\right.\ee
Similar to  Lemma \ref {2.6},  for each triplet $(M,t_0,y_0)\in
[0,\infty)\times [0,T)\times L^2(\om)$, Equation (\ref{tpbv}) has a
unique solution in $C([0,T];L^2(\om))$. Throughout this section,
\begin{itemize}
\item  $(\bar y^M_{t_0,y_0},\bar\psi^M_{t_0,y_0})$ denotes  the
 solution of (\ref{tpbv}) corresponding to $M$,  $t_0$ and $y_0$;
\item $\bar y^M_{t_0,y_0}$ and $\bar\psi^M_{t_0,y_0}$ denote
the first and the second component of the above solution, respectively, when one of them appears alone.

 \end{itemize}
Because of the assumption (\ref{1.2}), $\bar\psi^{N(t_0,y_0)}_{ t_0,
y_0}(T)=-(\bar y^{N(t_0,y_0)}_{ t_0, y_0}(T)-z_d)\neq 0$ (see the proof of Lemma~\ref{lemma2.1}). Thus, it
follows from the unique continuation property of the heat equation
(see \cite{FHL}) that
\begin{eqnarray}\label{uniqueFHL}{\chi}_{\omega}\bar\psi^{M(t_0,y_0)}_{t_0,y_0}(t_0)\neq
0\;\;\mbox{for all}\;\; (t_0,y_0)\in [0,T)\times L^2(\om).
\end{eqnarray}
\noindent Now we define a feedback law $F:[0,T)\times
L^2({\Omega})\mapsto L^2({\Omega})$ by setting
\bb\label{feedback}
\ds
F(t_0,y_0)=N(t_0,y_0)\frac{{\chi}_{\omega}\bar\psi^{N(t_0,y_0)}_{t_0,y_0}(t_0)}{\|{\chi}_{\omega}\bar\psi^{N(t_0,y_0)}_{
t_0, y_0}(t_0)\|},\; (t_0, y_0)\in[0,T)\times L^2({\Omega}).
\ee
Because of the existence and uniqueness of the solution to
(\ref{tpbv}), as well as  (\ref{uniqueFHL}), the map $F$ is
well defined. For each $(t_0,y_0)\in [0,T)\times L^2(\om)$, consider  the evolution equation:
\begin{equation}\label{feedbacksystem}
\left\{\begin{array}{ll} \ns\dot{y}(t)-A
y(t)=\chi_{\omega}F(t,y(t)),\qq&t\in(t_0,T),\\
\nm y(t_0)=y_0,
\end{array}\right. \end{equation}
where the operator $A$ was defined in Section 1.
Two main results in this section   are as follows:

\begin{Theorem} \label{theorem3.9}  For each  pair $(t_0, y_0)\in[0,T)\times
L^2({\Omega})$,  Equation (\ref{feedbacksystem}) has a unique (mild) solution. Furthermore,
this solution is exactly $\bar y^{N(t_0,y_0)}_{t_0,y_0}(\cdot)$.
\end{Theorem}

\begin{Theorem}\label{feedbacktheorem} For each pair  $(t_0,y_0)\in
[0,T)\times L^2(\om)$, $F(\cdot, y_F(\cdot;t_0,y_0))$ is the optimal
control to $(NP)_{t_0,y_0}$, where $y_F(\cdot;t_0,y_0))$ is the unique solution to  Equation (\ref{feedbacksystem}) corresponding to the initial data $(t_0,y_0)$.
\end{Theorem}

\medskip

It follows directly from  Theorem~\ref{theorem3.9} that for each $y_0\in L^2(\Omega)$ and each $\tau\in[0,T)$,
the  evolution equation
\begin{equation}\label{feedbacksystem2}
\left\{\begin{array}{ll} \ns\dot{y}(t)-A
y(t)=\chi_{\omega}\chi_{(\tau,T)}F(t,y(t)),\qq&t\in (0,T),\\
\nm y(0)=y_0
\end{array}\right. \end{equation}
admits a unique (mild) solution, denoted by $y_{F,\tau,y_0}(\cdot)$. Thus,  the following result is
a direct consequence of   Theorem~\ref{feedbacktheorem}:

\begin{Corollary} For each $y_0\in L^2(\om)$ and each $\tau\in [0,T)$, $ \chi_{(\tau,T)}(\cdot)F( \cdot,y_{F,\tau,y_0}(\cdot))$ is the optimal control to
Problem $(NP)^{r,\tau}$ with the initial state $y_0$.\end{Corollary}

\subsection{ Proof of Theorem~\ref{theorem3.9} (Part 1): The existence of solutions}

$\;\;\;\;$By a very similar argument to prove  Lemma \ref {lemma2.6}, we can obtain that
\begin{equation}\label{openloop}
 \ds\bar u^M_{t_0,y_0}(t)= M\frac{\chi_{\omega}\bar\psi^M_{t_0,y_0}(t)}
{||\chi_{\omega}\bar\psi^M_{t_0,y_0}(t)||},\;
t\in[t_0,T).\end{equation}
By the uniqueness and existence of the
solution to (\ref{tpbv}),  we can easily derive the following
consequence, which, in some sense, is a dynamic programming
principle.

\medskip

\begin{Lemma} \label {DGC}Let
$(t_0,y_0)\in[0,T)\times L^2({\Omega})$ and $ M\ge 0$. Then, for each
$s\in(t_0,T)$,
\begin{equation} \label {DG1}\ds( \bar y^M_{t_0,y_0},\bar\psi^M_{t_0,y_0}){\bigg
|}_{[s,T]}
=\bigr (\bar  y^M_{s,\bar y^M_{t_0,y_0}(s)},\bar\psi^M_{s,\bar
y^M_{t_0,y_0}(s)}\bigl). \ee \end{Lemma}
\begin{Lemma} \label{feedback-1} Let
$(t_0,y_0)\in[0,T)\times L^2({\Omega})$. Then
 \begin{equation}\label{normcondition}
M=N(t_0,y_0)\;\;\mbox{if and only if}\;\; \ds\left\|\bar
 y^M_{t_0,y_0}(T)-z_d\right\|=r\wedge\left\|e^{(T-t_0)\triangle}y_0-z_d\right\|,
 \end{equation}
 where "$\wedge$" is the symbol  taking the smaller. Moreover,
 the control, defined by
\begin{equation}\label{op}
 \ds\bar u^{N(t_0,y_0)}_{t_0,y_0}(t)= N(t_0,y_0)\frac{\chi_{\omega}\psi^{N(t_0,y_0)}_{t_0,y_0}(t)}
{||\chi_{\omega}\psi^{N(t_0,y_0)}_{t_0,y_0}(t)||},\;
 t\in (t_0,T),\end{equation}
 is the unique optimal control of Problem $(NP)_{t_0,y_0}$.\end{Lemma}

\begin{proof} First, we  show  (\ref{normcondition}) for  the case
where $\|e^{(T-t_0)\Delta} y_0-z_d\|>r$. In this case,
we can apply    the analogous
version of Proposition~\ref{theorem3.1} for Problem $(NP)_{t_0,y_0}$ to get that
$M=N(t_0,y_0)$ if and only if $\|\bar y^M_{t_0,y_0}(T)-z_d\|=r$.
This leads to (\ref{normcondition}) for this case.

Next, we prove   (\ref{normcondition}) for  the case where  $\|e^{(T-t_0)\Delta }y_0-z_d\|\leq r$. In this case, one can
easily check that $N(t_0,y_0)=0$,  the null control is the optimal
control to $(OP)^0_{t_0,y_0}$, and  $\bar
y^0_{t_0,y_0}(\cdot)=y(\cdot;0,t_0,y_0)=e^{(\cdot-t_0)\Delta}y_0$ over $[t_0,T]$.
Suppose that  $M=N(t_0,y_0)$. Then it holds that $M=0$ and
 $\|\bar
y^0_{t_0,y_0}(T)-z_d\|=\|e^{(T-t_0)\Delta}y_0-z_d\|$. These lead to the statement on the right hand side of (\ref{normcondition}). Conversely,
suppose that there is an $M_0\geq 0$ such that
\begin{eqnarray}\label{NORMA}
\|\bar y^{M_0}_{t_0,y_0}(T)-z_d\|=\|e^{(T-t_0)\Delta}y_0-z_d\|.
\end{eqnarray}
To show the statement on the left side of (\ref{normcondition}), it suffices to prove that $M_0=0$.
By the analogous version of Lemma \ref{lemma2.6} for
$(OP)^{M_0}_{t_0,y_0}$ (see (\ref{2.10})), it holds that
\begin{eqnarray}\label{NORMB}
\|\bar y^{M_0}_{t_0,y_0}(T)-z_d\|=r_{t_0,y_0}(M_0),
\end{eqnarray}
where $r_{t_0,y_0}(\cdot)$ corresponds to the map $M\rightarrow
r(0,M)$ given in Section 1, namely,
$$
r_{t_0,y_0}(M)= \inf\{ \|y(T; u,t_0,y_0)-z_d\| : u\in
B(0,M))\}, \; M\geq 0.
$$
Since the null control is the optimal control to $(OP)^0_{t_0,y_0}$,
we find that
$$
r_{t_0,y_0}(0)=\|y(T;0,t_0,y_0)-z_d\|=\|e^{(T-t_0)\Delta}y_0-z_d\|.
$$
Along with (\ref{NORMA}) and  (\ref{NORMB}), this indicates that
\begin{eqnarray}\label{NORMC}
r_{t_0,y_0}(0)=r_{t_0,y_0}(M_0).
\end{eqnarray}
By the analogous version of Lemma~\ref{lemma3.1} for
$(OP)^M_{t_0,y_0}$, the map $M\rightarrow r_{t_0,y_0}(M)$ is
strictly monotonically decreasing. This, together with
(\ref{NORMC}), yields that $M_0=0$.

 In summary, we conclude that
(\ref{normcondition}) stands.

Finally, we prove (\ref{op}). In the case that $\|e^{(T-t_0)}\Delta
y_0-z_d\|>r$,  according to the analogous version of
Proposition~\ref{theorem3.1} for Problem $(NP)_{t_0,y_0}$, the control
defined by (\ref{op})  is the unique optimal control of Problem
$(NP)_{t_0,y_0}$. In the case where  $\|e^{(T-t_0)}\Delta
y_0-z_d\|\leq r$, it is clear that $N(t_0,y_0)=0$ and the null
control is the optimal control to $(NP)_{t_0,y_0}$. Hence,
(\ref{op}) holds for this case. This completes the proof.

\end{proof}

The following result shows that the functional $N(\cdot,\cdot)$
holds the dynamic programming principle.

\begin{Lemma} \label{NDP}Let $(t_0,y_0)\in[0,T)\times L^2({\Omega})$.
Then it stands that
 \bb\label {DG2}
 \ds N(t_0,y_0)=N\biggr(s, ~\bar y^{N(t_0,y_0)}_{t_0,y_0}(s)\biggl) \; \mbox{for each} \;  s\in(t_0,T).
\ee \end{Lemma}

\noindent{\it Proof.} In the case where %
$e^{(T-t_0)\triangle}y_0\in B(z_d,r)$,
it is clear that
 $$N(t_0,y_0)=0 \;\;\mbox{and}\;\;
 \bar y^{N(t_0,y_0)}_{t_0,y_0}(\cdot)=e^{(\cdot-t_0)\triangle}y_0.$$
Because
 $$e^{(T-s)\triangle}\bar y^{N(t_0,y_0)}_{t_0,y_0}(s)=e^{(T-s)\triangle}e^{(s-t_0)\triangle}y_0=e^{(T-t_0)\triangle}y_0\in B(z_d,r)\;\;\mbox{for each}\;\; s\in (t_0,T),$$
we see that  $N\bigr(s, ~\bar
 y^{N(t_0,y_0)}_{t_0,y_0}(s)\bigl)=0$. Therefore the equality (\ref{DG2}) holds for this case.

 \medskip

  In the case where  $e^{(T-t_0)\triangle}y_0\notin
 B(z_d,r)$, it is clear that  $r\wedge\left\|e^{(T-t_0)\triangle}y_0-z_d\right\|=r$. By
 the analogous version of
Proposition~\ref{theorem3.1} for Problem $(NP)_{t_0,y_0}$, it holds that $\|\bar
 y^{N(t_0,y_0)}_{t_0,y_0}(T)-z_d\|=r$. Thus,  it follows from (\ref{normcondition}) that
\begin{eqnarray}\label{DG21-a}
\bar y^{N(t_0,y_0)}_{t_0,y_0} (T)\in \partial B(z_d,r).
\end{eqnarray}
We  claim that
 \bb\label {DG21}
 e^{(T-s)\triangle}\bar y^{N(t_0,y_0)}_{t_0,y_0}(s)\notin B(z_d,r)\;\;\mbox{for all}\;\; s\in(t_0,T).
 \ee
If (\ref{DG21}) did not hold, then  there would  exist a $\hat s\in(t_0,T)$
such that
\begin{eqnarray}\label{HK1}
 e^{(T-\hat s)\triangle}\bar y^{N(t_0,y_0)}_{t_0,y_0}(\hat s)\notin B(z_d,r).
\end{eqnarray}
We construct a
control $\hat u$ by setting
\begin{eqnarray}\label{HK2}\hat u(s)=\ds \chi_{(t_0,\hat s)}(s) N(t_0,y_0)\frac{\chi_{\omega}\bar\psi^{N(t_0,y_0)}_{t_0,y_0}(s)}
{\|\chi_{\omega}\bar\psi^{N(t_0,y_0)}_{t_0,y_0}(s)\|},\;\; s\in [t_0,T).
\end{eqnarray}
Clearly, the solution $y(\cdot; \hat u, t_0,y_0)$ to (\ref{state}), where $u=\hat u$, coincides with
$\bar y^{N(t_0,y_0)}_{t_0,y_0} (\cdot)$ over $[t_0,\hat s]$. This, along with (\ref{HK2}) and (\ref{HK1}), indicates that
\begin{eqnarray}\label{HK3}
 y(T; \hat u,t_0,y_0)= e^{(T-\hat s)\triangle}y(\hat s;\hat u,t_0,y_0)\in B(z_d,r).
\end{eqnarray}
On the other hand, it follows from (\ref{HK2}) that $\ds\|\hat u||_{L^{\infty}(t_0,T;L^2(\Omega))}=N(t_0,y_0).$
This, together with (\ref{HK3}), yields that
  $\hat u$ is the optimal control to $(NP)_{t_0,y_0}$.
 However,  the problem $(NP)_{t_0,y_0}$ holds  the bang-bang
property (it follows from the analogous version of  Proposition \ref{theorem3.3} for $(NP)_{t_0,y_0}$). This implies that   $\|\hat u(s)\|=N(t_0,y_0)$ for a.e. $s\in (t_0,T)$, which
contradicts to the structure of $\hat u$. Hence, (\ref{DG21})
stands.

 Next,  by (\ref{DG1}) and (\ref{DG21-a}), we see
that
$$\ds\bar y^{N(t_0,y_0)}_{s,\bar y^{N(t_0,y_0)}_{t_0,y_0}(s)}(T)
=\bar y^{N(t_0,y_0)}_{t_0,y_0} (T)\in \partial B(z_d,r)\;\mbox{for each}\; s\in (t_0,T).$$
This, together with (\ref{DG21}), implies that
\begin{eqnarray}\label{HK4}
\|\ds\bar y^{N(t_0,y_0)}_{s,\bar y^{N(t_0,y_0)}_{t_0,y_0}(s)}(T)-z_d\|=r\wedge \|e^{(T-s)\triangle}\bar y^{N(t_0,y_0)}_{t_0,y_0}(s)-z_d\|\; \mbox{for each}\; s\in (t_0,T).
\end{eqnarray}
Now, we arbitrarily fix a $s\in (t_0,T)$. By (\ref{HK4}), we can apply  (\ref{normcondition}), with $t_0=s$ and
$y_0=\bar y^{N(t_0,y_0)}_{t_0,y_0}(s)$, to get that
$$
N(t_0,y_0)=N(s, \bar y^{N(t_0,y_0)}_{t_0,y_0}(s)),
$$
which  gives
 the equality (\ref{DG2}) for the second case.
 In summary, we finish the proof.\endpf
\bigskip

\noindent{\bf Proof of Theorem~\ref{theorem3.9} (Part 1): The
existence.}  It follows from (\ref{feedback}) (the definition of $F$) that
\begin{eqnarray*}
\ds F(t,~\bar y^{N(t_0,y_0)}_{t_0,y_0}(t))=N(t,~\bar
y^{N(t_0,y_0)}_{t_0,y_0}(t)) \frac{\phi(t)}{\|\phi(t\|},\; t\in
(t_0,T),
\end{eqnarray*}
where $\phi(t)={\chi}_{\omega}\bar\psi^{N(t,~\bar
y^{N(t_0,y_0)}_{t_0,y_0}(t))}_{t,~\bar
y^{N(t_0,y_0)}_{t_0,y_0}(t)}(t )$. This, together with (\ref{DG2})
and (\ref{DG1}), yields that
$$\ds F(t,~\bar y^{N(t_0,y_0)}_{t_0,y_0}(t))
=N(t_0,y_0)\frac {{\chi}_{\omega}\bar\psi^{N(t_0,y_0)}_{t_0,y_0}(t)}
{\|{\chi}_{\omega}\bar\psi^{N(t_0,y_0)}_{t_0,y_0}(t)\|}
=\frac{d}{dt}\bar y^{N(t_0,y_0)}_{t_0,y_0}(t)
+A \bar y^{N(t_0,y_0)}_{t_0,y_0}(t),\; t\in (t_0,T).$$
Here, we used that $\chi_\omega\circ\chi_\omega=\chi_\omega$. From the above equality and the fact that $\bar y^{N(t_0,y_0)}_{t_0,y_0}(t_0)=y_0$ , it follows that $\bar y^{N(t_0,y_0)}_{t_0,y_0}(\cdot)$ is a
solution to (\ref{feedbacksystem}). This completes the proof. \endpf

\bigskip

\subsection{ Proof of Theorem~\ref{theorem3.9} (Part 2): The uniqueness}

$\;\;\;\;$The key to prove the uniqueness is showing the following properties of
the feedback law $F(\cdot,\cdot)$.
\begin{Proposition}\label{LIPOFF} $(i)$  For each pair   $ (\bar t_0,\bar y_0 )\in[0,T)\times L^2(\Omega)
$, there is a $\bar \rho>0$ such that $F(t_0,\cdot)$ is Lipschitz
continuous in $B(\bar y_0,\bar \rho)$  uniformly with respect to
$t_0\in [(\bar t_0-\bar\rho)^+,\bar t_0+\bar\rho]$. $(ii)$ For each
 $\bar y_0\in L^2(\om)$, $F(\cdot\,,\bar y_0)$ is continuous
over $[0,T)$.\end{Proposition} \noindent When it is proved, {\bf the
uniqueness of the solution to Equation (\ref{feedbacksystem})
follows immediately from the  generalized Picard-Lindelof Theorem
(see \cite{Zeidler}) and Proposition~\ref{LIPOFF}. Consequently, the
proof of Theorem~\ref{theorem3.9} is completed. }\\

The remainder is showing Proposition~\ref{LIPOFF}. To serve such purpose,
we first study some continuity properties of
$N(\cdot,\cdot)$. These properties will be concluded in Lemma~\ref{Lip1}.
Two lemmas before it will play important roles in its proof.

\begin{Lemma} \label{convex} For each $t_0\in [0,T)$, the functional  $N(t_0,\cdot)$ is
convex over $L^2(\om)$. \end{Lemma}

\noindent{\it Proof.} Let $y_0^1$ and $y_0^2$ belong to $L^2(\Omega)$. The optimal
controls $\bar u^i$ to $(NP)_{t_0,y_0^i}$, $i=1,2$, satisfy that
$N(t_0,y_0^i)=\|\bar u^i\|_{L^\infty(t_0,T;L^2(\om))}$ and $y(T;\bar u^i,t_0,y_0^i)\in B(z_d,r)$, $ i=1,2.$
Since  for each $ \l\in(0,1)$,
$$y(T;\l \bar u^1+(1-\l)\bar u^2,t_0,\l y_0^1+(1-\l)y_0^2)=\l y(T;\bar
u^1,t_0,y_0^1)+(1-\l)y(T;\bar u^2,t_0,y_0^2)\in B(z_d,r),$$
we obtain that  $$\ba{l }\ns
N(t_0,\l y_0^1+(1-\l)y_0^2)
\le\|\l \bar u^1+(1-\l)\bar u^2\|_{L^\infty(t_0,T;L^2(\om))}\\
\ns\le \l\|\bar u^1\|_{L^\infty(t_0,T;L^2(\om))}+(1-\l)\|\bar u^2\|_{L^\infty(t_0,T;L^2(\om))}\\
\ns= \l N(t_0,y_0^1)+(1-\l)N(t_0,y_0^2).\ea.$$
This completes the proof.\endpf
\bigskip

\begin{Lemma} \label{M-bounded} For each  $\bar t_0\in[0,T)$ and each bounded subset  $E$ of $L^2(\Omega)$, the functional  $N(\cdot,\cdot)$ is bounded
on $\Bigr[(\bar t_0-\bar \d)^+,\bar t_0+\bar\d\Bigl]\times E$, where
$\bar\d=(T-\bar t_0)/2$.\end{Lemma}

\noindent{\it Proof.} Write $C_E=\sup\{\|y_0\| : y_0\in E \}$. Let
$(t_0,y_0)\in \Bigr[(\bar t_0-\bar \d)^+,\bar t_0+\bar\d\Bigl]\times
E$. By the null controllability of the heat equation over $(t_0,
\bar t_0+3\bar\d/2)$ (see, for instance, \cite{EZ}), there is a
control $u_1$ with
\bb\label{estimate}
 \| u_1\|_{L^\infty(t_0, \bar t_0+3\bar\d/2; L^2(\om))}\le C_1\|y_0\|\leq C_1C_E,\ee
where $C_1>0$ is independent of $t_0$ and $y_0$, such that
$$ \bar y
\equiv y(\bar t_0+3\bar\d/2; u_1,t_0,y_0)=0.$$ Here we used that
$t_0\leq\bar t_0+\bar \delta$. Then, by the approximate
controllability of the heat equation over $(\bar t_0+3\bar\d/2, T)$
(see, for instance, \cite{FPZ}), there is another control $u_2$ with
$$
 \| u_2\|_{L^\infty(\bar t_0+3\bar\d/2, T ; L^2(\om))}\le C_2,
 $$
where $C_2>0$ is independent of $t_0$ and $y_0$, such that $$y(T;
u_2, \bar t_0+3\bar\d/2, \bar y)\in B(z_d,r).$$ Clearly, the control
 $v\equiv\chi_{(t_0, \bar t_0+3\bar\d/2)}u_1+\chi_{(\bar t_0+3\bar\d/2,
T)}u_2$ satisfies that $y(T;v,t_0,y_0)\in B(z_d,r)$. Therefore, it
holds that
$$
N(t_0,y_0)\leq \|v\|_{L^\infty(t_0,T;L^2(\om))}\leq \max\{ C_1C_E,
C_2\}.
$$
This completes the proof.
\endpf

\begin{Lemma} \label{Lip1} $(i)$ For each    $(\bar t_0, \bar y_0)\in[0,T)\times L^2(\om)$ and each $\rho\in(0,1/2)$,  $N(t_0,\cdot)$ is
Lipschitz continuous over $B(\bar y_0,\rho)$ uniformly w.r.t.
$t_0\in\Bigr[(\bar t_0-\bar\d)^+,\bar t_0+\bar \d\Bigl]$, where
$\bar\d=(T-\bar t_0)/2$; $(ii)$ For each $\bar y_0\in L^2(\om)$,
$N(\cdot,\bar y_0)$ is continuous over $[0,T)$.\end{Lemma}

 \noindent{\it Proof.} $(i)$ Let $(\bar t_0, y_0)\in[0,T)\times
 L^2(\om)$ and let $\rho\in(0,1/2)$. We arbitrarily take two different points  $y_0^1$ and $y_0^2$ from $
B(\bar y_0,\rho)\subset B(\bar y_0,1)$. Denote by $\mathcal{L}$ the
straight line passing through $y_0^1$ and $y_0^2$, namely,
$\mathcal{L}\equiv \biggr\{y_0^\l\deq(1-\l) y_0^1+\l y_0^2 \bigm|\l\in(-\infty,+\infty)\biggl\}$.
Clearly, $\mathcal{L}$ intersects with $B(\bar y_0,1)$ at two
different points, denoted by $y_0^{\l_1}$ and $y_0^{\l_2}$, with
$\l_1 <\l_2$.
 Since the  segment
 $\{\,y_0^\l|\l\in[0,1]\,\}\subseteq B(\bar y_0,\rho)$ and $B(\bar y,\rho)\bigcap\partial B(\bar
 y,1)=\emptyset$, it holds that
 $\l_1<0<1<\l_2.$  Moreover, one can easily check that
 \bb\label{proportion}
 \ds\frac{\|y_0^1-y_0^{\l_1}\|}{0-\l_1}=\frac{\|y_0^2-y_0^1\|}{1-0}=\frac{\|y_0^{\l_2} -y_0^2\|}{\l_2-1}.
\ee

\noindent For each $t_0\in \Bigr[(\bar t_0-\bar\d)^+,\bar t_0+\bar
\d\Bigl]$, we define  a function $g_{t_0}(\cdot;y_0^1,y_0^2)$ by setting
$$g_{t_0}(\l; y_0^1,y_0^2)=N(t_0,(1-\l) y_0^1+\l y_0^2),\qq \lambda\in(-\infty,+\infty).$$
Obviously, the convexity of $N(t_0,\cdot)$ (see Lemma \ref{convex})
implies the convexity of $g_{t_0}(\cdot; y_0^1,y_0^2)$. By the property of
convex functions, one has that
$$\ds\frac{g_{t_0}(0 ;y_0^1,y_0^2)-g_{t_0}(\l_1;y_0^1,y_0^2)}{0-\l_1}\le\frac{g_{t_0}(1;y_0^1,y_0^2)-g_{t_0}(0;y_0^1,y_0^2)}
{1-0}\le\frac{g_{t_0}(\l_2;y_0^1,y_0^2)-g_{t_0}(1;y_0^1,y_0^2)}{\l_2-1}.$$
This, along with (\ref{proportion}) and the nonnegativity of
$g_{t_0}(1;y_0^1,y_0^2)$, indicates  that
\begin{eqnarray}\label{proposition-a}
\ba{l}
\ns\ds~~~\frac{N(t_0,y_0^2)-N(t_0,y_0^1)}{\|y_0^2-y_0^1\|}=\frac{g_{t_0}(1;y_0^1,y_0^2)-g_{t_0}(0;y_0^1,y_0^2)}
{\|y_0^2-y_0^1\|}\\

\nm\ds\le
\frac{g_{t_0}(\l_2;y_0^1,y_0^2)-g_{t_0}(1;y_0^1,y_0^2)}{(\l_2-1)\|y_0^2-y_0^1\|}\le\frac{g_{t_0}(\l_2;y_0^1,y_0^2)}
{(\l_2-1)\|y_0^2-y_0^1\|}\\

\ns\ds=\frac{g_{t_0}(\l_2;y_0^1,y_0^2)}{\|y_0^{\l_2}-y_0^2\|}. \ea
\end{eqnarray}
Two observations are as follows. The triangle inequality implies
that
$$\|y_0^{\l_2}-y_0^2\|\ge\|y_0^{\l_2}-\bar y_0\|-\|y_0^2-\bar y_0\|\ge 1-\rho;$$
The boundedness of $N(\cdot,\cdot)$ (\see Lemma \ref{M-bounded}) gives
that for each $t_0\in \Bigr[(\bar t_0-\bar\d)^+,\bar t_0+\bar \d\Bigl]$, $y_0^1, y_0^2\in B(\bar y_0,\rho)$,
$$g_{t_0}(\l_2;y_0^1,y_0^2)=N(t_0,y_0^{\l_2})\le\sup\biggr\{N(s_0,z_0)\Bigm|(s_0,z_0)\in
\Bigr[(\bar t_0-\bar\d)^+,\bar t_0+\bar\d\Bigl]\times B(\bar
y_0,1)\biggl\}\equiv C.$$

Along with these two observations, (\ref{proposition-a}) yields that
 $${N(t_0,y_0^1)-N(t_0,y_0^2)}\le \frac{C}{1-\rho}{\|y_0^2-y_0^1\|}\equiv C(\rho){\|y_0^2-y_0^1\|}.$$
 Similarly, we can obtain  ${N(t_0,y_0^2)-N(t_0,y_0^1)}\le
 C(\rho){\|y_0^2-y_0^1\|}.$ These lead to the desired Lipschitz continuity.

\medskip

\noindent $(ii)$ Let $\bar y_0\in L^2(\om)$. Arbitrarily take $\bar
t_0$ from $[0,T)$ and write $\bar\d=(T-\bar t_0)/2$. It suffices to
show that $N(\cdot, \bar y_0)$ is continuous over $\Bigr[(\bar
t_0-\bar\d)^+,\bar t_0+\bar\d\Bigl]$. For this purpose, we
arbitrarily take two
 different $t_0^1$ and $t_0^2$ from this interval. Without lose of generality, we can  assume that
$t_0^1<t_0^2$. Then by (\ref{DG2}) (see Lemma \ref{NDP}), the part
$(i)$ of the current lemma and Lemma \ref{M-bounded}, we can easily deduce that
$$\ba{l}
\ns~~~|N(t_0^1,\bar y_0)-N(t_0^2,\bar y_0)|\\
 \ns=\left|N\Bigr(t_0^2,\,\bar y^{N(t_0^1,\bar y_0)}_{t_0^1,\bar
 y_0}(t_0^2)\Bigr)
 -N(t_0^2,\bar y_0) \right|
 \le C\left\|\bar y^{N(t_0^1,\bar y_0)}_{t_0^1,\bar y_0}(t_0^2)-\bar y_0)\right\|\\
 \ns\ds=C\left\|\Bigr[e^{(t_0^2-t_0^1)\triangle}-I\Bigl]\bar y_0+\int^{t_0^2}_{t_0^1}e^{(t_0^2-s)\triangle}
 \bar u^{N(t_0^1,\bar y_0)}_{t_0^1,\bar y_0}(s)ds\right\|\\
\ns\ds\le C\left\|e^{(t_0^2-t_0^1)\triangle}-I\right\|\|\bar
 y_0\|+ C|t_0^2-t_0^1|
\ea$$
where $C$ stands for a positive constant independent of $t_0^1$ and
$t_0^2$. It varies in different contexts. Clearly, the continuity of
$N(\cdot, \bar y_0)$ over $\Bigr[(\bar t_0-\bar\d)^+,\bar
t_0+\bar\d\Bigl]$ follows from the above inequality at once.

In summary, we finish the proof.\endpf

\bigskip

Next, we study some  properties for the map $\cN: [0,T)\times
L^2(\Omega)\times[0,+\infty)\mapsto L^2(\Omega)$ defined by
\begin{eqnarray}\label{MAPCN}
\cN(t_0,y_0,M)= \bar
u^M_{t_0,y_0}(t_0)=M\frac{\chi_{\omega}\bar\psi^M_{t_0,y_0}(t_0)}{\|\chi_{\omega}\bar\psi^M_{t_0,y_0}(t_0)\|}.
\end{eqnarray}
\vskip 10pt

\begin{Lemma}\label{Lip2} $(i)$ For each  $ (\bar t_0,\bar y_0,\bar M )\in[0,T)\times
L^2(\Omega)\times[0,\infty)$, there is a $\bar\rho>0$ such that  $\cN(t_0,\cdot,\cdot)$ is Lipschitz continuous over $B(\bar
y_0, \bar \rho)\times [(\bar M-\bar \rho)^+, \bar M+\bar \rho]$ uniformly with respect to  $t_0\in B(\bar
t_0, \bar\rho)\bigcap[0,T)$. $(ii)$ $\cN(\cdot,\bar y_0,\cdot)$ is
continuous over $[0,T)\times [0,\infty)$. \end{Lemma}

\noindent{\it Proof.} $(i)$ Let $ (\bar t_0,\bar y_0,\bar M
)\in[0,T)\times L^2(\Omega)\times[0,\infty)$. The proof of the first
continuity will be carried by several steps as follows:\\

\noindent{\it Step 1.   For all $t_0\in [0,T)$,  $0\le M_1\le M_2$ and $y_0^1,~y_0^2\in
L^2(\om)$, it holds that
\bb\label{estimate2}\ba{rl}
 \ns&\ds\|\cN(t_0,y_0^1, M_1)-\cN(t_0,y_0^2,M_2)\|\leq|M_1-M_2|+\\
\nm &\ds\frac{4M_1}{\| \chi_{\omega}\bar \psi^{M_1}_{t_0,y^1_0} \|}
\left[M_1\|y_0^1-y_0^2\|+(\|y_0^2\|+\|z_d\|){|M_1-M_2|}\right]\;\;\mbox{when}\;M_1>0;
\ea \ee
\bb\label{estimate2-1}
 \|\cN(t_0,y_0^1, M_1)-\cN(t_0,y_0^2,M_2)\|=|M_1-M_2|\;\;\mbox{when}\;  M_1=0,\ee
}

The equality (\ref{estimate2-1}) follows directly  from the
definition of $\cN$. Now we prove (\ref{estimate2}). For
simplification of notation, we write
$$\bar y^i\deq\bar y^{M_i}_{t_0,y^i_0},\q\bar\psi^i\deq\bar
\psi^{M_i}_{t_0,y^i_0},\q \bar u^i\deq\bar u^{M_i}_{t_0,y^i_0},\qq
i=1,2.
$$
It is clear that  that
$(1-\e)\bar u^1+\e \frac{M_1}{M_2}\bar u^2\in L^\infty(t_0,T;B(0,M_1))$
for any $\e\in[0,1],$ which, together with the optimality of $\bar
u^1$ to $(OP)^{M_1}_{t_0,y_0^1}$, shows  that
$$\ba{ll}
\ns0&\ds\le \liminf\limits_{\e\rightarrow0+}\frac{1}{2\e}\left\{
 \biggr\|e^{(T-t_0)\triangle}y_0^1+\int^T_{t_0}e^{(T-s)\triangle}[(1-\e)\bar u^1+\e\frac{M_1}{M_2}\bar
 u^2]ds-z_d\biggl\|^2\right.\\
 \ns&~~~~~~~~~~~~\left.\ds-\biggr\|e^{(T-t_0)\triangle}y_0^1+\int^T_{t_0}e^{(T-s)\triangle}\bar
 u^1ds-z_d\biggl\|^2\right\}\\
\nm&\ds=\left\langle
e^{(T-t_0)\triangle}y_0^1+\int^T_{t_0}e^{(T-s)\triangle}\bar
 u^1ds-z_d\,,~\int^T_{t_0}e^{(T-s)\triangle}\left[\frac{M_1}{M_2}\bar
 u^2-\bar u^1\right]ds\right\rangle.
 \ea$$
Similarly, we can prove that
 $$\ds 0\le \left\langle e^{(T-t_0)\triangle}y_0^2+\int^T_{t_0}e^{(T-s)\triangle}\bar
 u^2ds-z_d\,,~\int^T_{t_0}e^{(T-s)\triangle}\left[\frac{M_2}{M_1}\bar
 u^1-\bar u^2\right]ds\right\rangle.$$
Dividing the  first inequality above by $M_1^2$  and  the second one
by $M_2^2$, then adding them together, we obtain that
 $$\ba{l}
\ds~~~\left\langle
\frac{e^{(T-t_0)\triangle}y_0^1-z_d}{M_1}
-\frac{e^{(T-t_0)\triangle}y_0^2-z_d}{M_2}\,,~
 \int^T_{t_0}e^{(T-s)\triangle}\left[ \frac{\bar u^2}{M_2}
 -\frac{\bar  u^1}{M_1}\right]ds\right\rangle\\
 \ge\ds \biggr\|\int^T_{t_0}e^{(T-s)\triangle}\left[ \frac{\bar u^2}{M_2}
 -\frac{\bar  u^1}{M_1}\right]ds\biggl\|^2,
 \ea
$$
 which implies that
\begin{equation}\label{estimate4}\biggr\|\frac{e^{(T-t_0)\triangle}y_0^1-z_d}{M_1}
-\frac{e^{(T-t_0)\triangle}y_0^2-z_d}{M_2}\biggl\|
 \ge\biggr\|\int^T_{t_0}e^{(T-s)\triangle}\left[ \frac{\bar u^2}{M_2}
 -\frac{\bar  u^1}{M_1}\right]ds\biggl\|.\end{equation}
Since $(\bar y^i, \bar \psi^i)$, $i=1,2$,  solve (\ref{tpbv}) and the semigroup  $\{
e^{t\triangle}: t\geq 0\}$ is contractive, we can use (\ref{estimate4}) to derive that\\
 \bb\label{estimate3}\ba{rl}
\ns&\ds\biggr\|\frac{\bar\psi^1(t_0)}{M_1}
-\frac{\bar\psi^2(t_0)}{M_2}\biggl\|
=\biggr\|e^{(T-t_0)\triangle}\left(\frac{\bar\psi^1(T)}{M_1}
-\frac{\bar\psi^2(T)}{M_2}\right)\biggl\|\\

\ns\le&\ds\biggr\|\frac{\bar\psi^1(T)}{M_1}
-\frac{\bar\psi^2(T)}{M_2}\biggl\|
=\biggr\|\frac{\bar y^1(T)-z_d}{M_1}
-\frac{\bar y^2(T)-z_d}{M_2}\biggl\|\\

\ns\le&\ds\left\| \frac{e^{(T-t_0)\triangle}y^1_0-z_d}{M_1}
-\frac{e^{(T-t_0)\triangle}y^2_0-z_d}{M_2}\right\|
 +\left\| \int^T_{t_0}e^{(T-s)\triangle}\left[\frac{\bar u^1}{M_1}
 -\frac{\bar u^2}{M_2}\right]ds\right\|\\

\ns\le&\ds \frac{2}{M_1} \|y_0^1-y_0^2\|
+2(\|y_0^2\|+\|z_d\|)\frac{|M_1-M_2|}{M_1M_2}.
\ea\ee

\noindent By direct computation, we obtain that
$$\ba{ll}
 \ns&\ds\|\cN(t_0,y_0^1, M_1)-\cN(t_0,y_0^2,M_2)\|
 =\biggr\|M_1\frac{\chi_{\omega}\bar\psi^1(t_0)}{\|\chi_{\omega}\bar\psi^1(t_0)\|}
-M_2\frac{\chi_{\omega}\bar\psi^2(t_0)}{\|\chi_{\omega}\bar\psi^2(t_0)\|}\biggl\|\\
\ns\le&\ds~M_1\biggr\|\frac{\chi_{\omega}\bar\psi^1(t_0)}{\|\chi_{\omega}\bar\psi^1(t_0)\|}
-\frac{\chi_{\omega}\bar\psi^2(t_0)}{\|\chi_{\omega}\bar\psi^2(t_0)\|}\biggl\|
+|M_1-M_2|\biggr\|\frac{\chi_{\omega}\bar\psi^2(t_0)}{\|\chi_{\omega}\bar\psi^2(t_0)\|}\biggl\|\\
\ns=&\ds\frac{M_1}{\|{\frac{\chi_{\omega}\bar\psi^1(t_0)}{M_1}}\|\|{\frac{\chi_{\omega}\bar\psi^2(t_0)}{M_2}}\|}
{\biggr\|}\left[\|\frac{\chi_{\omega}\bar\psi^2(t_0)}{M_2}\|\frac{\chi_{\omega}\bar\psi^1(t_0)}{M_1}
-\|\frac{\chi_{\omega}\bar\psi^1(t_0)}{M_1}\|\frac{\chi_{\omega}\bar\psi^2(t_0)}{M_2}\right]\biggl\|+|M_1-M_2|\\
\ns\le&\ds\frac{2M_1}{\|{\frac{\chi_{\omega}\bar\psi^1(t_0)}{M_1}}\|\|{\frac{\chi_{\omega}\bar\psi^2(t_0)}{M_2}}\|}
\left\|\frac{\chi_{\omega}\bar\psi^2(t_0)}{M_2}\right\|\left\|\frac{\chi_{\omega}\bar\psi^1(t_0)}{M_1}-
\frac{\chi_{\omega}\bar\psi^2(t_0)}{M_2}\right\|+|M_1-M_2|\\
\ns\le&\ds\frac{2M_1^2}{\|\chi_{\omega}\bar\psi^1(t_0)\|}
\left\|\frac{ \bar\psi^1(t_0)}{M_1}-\frac{
\bar\psi^2(t_0)}{M_2}\right\|+|M_1-M_2|.
 \ea$$

\noindent This, together with (\ref{estimate3}), shows
(\ref{estimate2}).

\bigskip
\noindent{\it Step 2.    When $\bar M>0$, there is $\bar\rho>0$ such that for
each  $( t_0,y_0,M)\in [(\bar t_0-\bar \rho)^+,\bar t_0+\bar
\rho]\times B(\bar y_0,\bar \rho )\times [\bar M-\bar \rho, \bar M,\bar \rho]$,
\begin{eqnarray}\label{LOWERBOUNDED1}
\left\|\chi_{\omega}\bar\psi^{M}_{t_0,y_0}(t_0)\right\|\geq\frac{1}{2}
\left\|\chi_{\omega}\bar\psi^{\bar M}_{\bar t_0,\bar y_0}(\bar
t_0)\right\|>0.
\end{eqnarray}  }

The
second inequality in (\ref{LOWERBOUNDED1}) follows from (\ref{uniqueFHL}). The
first one will be proved by the following two cases:

\noindent{\it Case 1:  $t_0\leq \bar t_0$.} In this case, the following three estimates hold for all
$y_0\in L^2(\Omega)$ and $M\in [\bar M/2, (3\bar M)/2]$:
$$\ba{ll}
 \ns&\ds\left \|\bar\psi^M_{t_0,   y_0}(t_0)-\bar\psi^{\bar M}_{\bar   t_0,\bar y_0}
 (\bar  t_0)\right\|=\ds\left \|e^{(\bar t_0-t_0)\triangle}\bar\psi^M_{t_0,  y_0}(\bar t_0)-\bar\psi^{ \bar M}_{\bar t_0, \bar y_0}(\bar  t_0)
\right\|\\
  \nm\le&\left \|e^{(\bar t_0-t_0)\triangle}\right\|\cdot\left\|\bar\psi^M_{t_0,  y_0}(\bar t_0)-
\bar\psi^{\bar M}_{\bar t_0, \bar y_0}(\bar t_0)\right\|
+ \left \| e^{(\bar t_0-t_0)\triangle}-I \right\|\cdot
  \|\bar\psi^{\bar M}_{\bar t_0, \bar y_0}(\bar t_0)\|\\
  \ns\le& \left\|\bar\psi^M_{t_0,  y_0}(\bar t_0)-
\bar\psi^{\bar M}_{\bar t_0, \bar y_0}(\bar t_0)\right\|
+ \left \| e^{(\bar t_0-t_0)\triangle}-I \right\|\cdot
  \|\bar\psi^{\bar M}_{\bar t_0, \bar y_0}(\bar t_0)\|;
\ea
$$
(Here $I$ denotes the identity operator on $L^2(\om)$.)
$$\ba{ll}
\ns&\left\|\bar\psi^M_{t_0,  y_0}(\bar t_0)
- \bar\psi^{\bar M}_{\bar t_0, \bar y_0}(\bar t_0)\right\|
=\left\|\bar\psi^M_{\bar t_0,  \bar y^M_{t_0,y_0}(\bar t_0)}(\bar
t_0)
- \bar\psi^{\bar M}_{\bar t_0, \bar y_0}(\bar t_0)\right\|\\
 \ns\le&\ds M\left\|\frac{\bar\psi^M_{\bar t_0,  \bar y^M_{t_0,y_0}(\bar t_0)}(\bar
 t_0)}{M}
- \frac{\bar\psi^{\bar M}_{\bar t_0, \bar y_0}(\bar t_0)}{\bar
M}\right\|
+\frac{|M-\bar M|}{\bar M}\left\|\bar\psi^{\bar M}_{\bar t_0, \bar
y_0}(\bar t_0)\right\|\\
 \ns\le&\ds 2\left\|\bar y^M_{t_0,y_0}(\bar t_0)-\bar y_0\right\|
 + \left( 2\|\bar y_0\|+ 2\|z_d\|+ \left\|\bar\psi^{\bar M}_{\bar t_0, \bar
y_0}(\bar t_0)\right\|   \right)\frac{|M-\bar M|}{\bar M};
\ea$$
(Here, (\ref{DG1}) and (\ref{estimate3}) have been used.) and
$$\ba{ll}
~~~~&\ds\left\|\bar y^M_{t_0,y_0}(\bar t_0)-\bar y_0
\right\|=\left\|e^{(\bar t_0-t_0)\triangle} y _0-\bar y_0
+ \int^{\bar t_0}_{t_0} e^{(\bar t_0-s)\triangle} \bar u^M_{t_0,y_0}(s)ds\right\|\\
 \ns\le&\ds  \left\|e^{(\bar t_0-t_0)\triangle}\right\|\cdot\| y _0-\bar y_0\|
 +\left\|e^{(\bar t_0-t_0)\triangle}-I\right\|\cdot\|\bar y_0\|
+ \int^{\bar t_0}_{t_0}\left\| e^{(\bar t_0-s)\triangle}\right\|\cdot \left\|\bar u^M_{t_0,y_0}(s)\right\| ds\\
 \ns\le&\ds\| y _0-\bar y_0\|
 +\left\|e^{(\bar t_0-t_0)\triangle}-I\right\|\cdot\|\bar y_0\|
+M (\bar t_0-{t_0}).
\ea$$

\noindent Combining  the above-mentioned three inequalities together
leads to
\bb\label{psiC1}\ba{l}
 ~~\ds\left \|\bar\psi^M_{t_0,   y_0}(t_0)-\bar\psi^{\bar M}_{\bar
t_0,\bar y_0}(\bar t_0)\right\|
\le 2\| y _0-\bar y_0\|+2M (\bar t_0-{t_0})\\
 \ns \ds

 + \left\|e^{(\bar t_0-t_0)\triangle}-I\right\|\cdot\left(2\|\bar
 y_0\|+\|\bar\psi^{\bar M}_{\bar t_0, \bar y_0}(\bar t_0)\|\right)\\

\ns\ds
 + \left( 2\|\bar y_0\|+ 2\|z_d\|+ \left\|\bar\psi^{\bar M}_{\bar t_0, \bar
y_0}(\bar t_0)\right\|   \right)\frac{|M-\bar M|}{ \bar{ M}}.
  \ea\ee%
Clearly,  the right hand side of (\ref{psiC1}) is continuous with respect to
$(t_0,y_0,M)$. This, along with  the second inequality in
(\ref{LOWERBOUNDED1}), indicates that there exists a $\rho_1$ with
$$\ds 0<\rho_1<\frac{T-\bar t_0}{2}\wedge\frac{\bar M}{2},$$
such that for each $(t_0,y_0,M)\in [(\bar t_0-\rho_1)^+,\bar
t_0]\times B(\bar y_0,\rho_1)\times [\bar M- \rho_1, \bar M+\rho_1]$,
\bb\label{psibounded1}
 \left \|\bar\psi^M_{t_0,   y_0}(t_0)-\bar\psi^{\bar M}_{\bar t_0,\bar y_0}(\bar t_0)\right\|
\le\ds\frac{1}{2} \left\|\bar\psi^{\bar M}_{\bar t_0,\bar y_0}(\bar
t_0)\right\|,  \ee

\bigskip

\noindent {\it Case 2:  $t_0\geq\bar t_0$.}  In this case, the following two estimates hold for all $y_0\in L^2(\om)$ and $M\in [\bar M/2, (3\bar M)/2]$:
$$\ba{ll}
\ns&\ds\left\|\bar\psi^M_{t_0,  y_0}(  t_0)
- \bar\psi^{\bar M}_{\bar t_0, \bar y_0}(  t_0)\right\|\\
 \ns\le&\ds M\left\|\frac{\bar\psi^M_{t_0,  y_0}(   t_0) }{M}
- \frac{\bar\psi^{\bar M}_{\bar t_0, \bar y_0}(  t_0)}{\bar
M}\right\|
+\frac{|M-\bar M|}{\bar M}\left\| \bar\psi^{\bar M}_{\bar t_0, \bar
y_0}(  t_0) \right\|\\
\ns=&\ds M\left\|\frac{\bar\psi^M_{t_0,  y_0}  (  t_0)}{M}
- \frac{\bar\psi^{\bar M}_{ t_0,  \bar y^{\bar M}_{\bar t_0,\bar
y_0}( t_0)}( t_0)}{\bar M}\right\|
+\frac{|M-\bar M|}{\bar M}\left\| \bar\psi^{\bar M}_{\bar t_0, \bar
y_0}(  t_0) \right\|\\
 \ns\le&\ds 2\left\| y_0 -\bar y^{\bar M}_{\bar t_0,\bar
y_0}( t_0) \right\|
 + \left(  2\left\|\bar y^{\bar M}_{\bar t_0,\bar
y_0}( t_0)\right\|+ 2\|z_d\|+  \left\|\bar\psi^{\bar M}_{\bar t_0,
\bar y_0}( t_0)\right\| \right)\frac{|M-\bar M|}{\bar M}
\ea$$
(Here,  (\ref{DG1}) and (\ref{estimate3}) have been used.) and
$$\ba{ll}
~~~~&\ds\left\|\bar y^{\bar M}_{\bar t_0,\bar y_0}(  t_0)-  y_0
\right\|=\left\|e^{( t_0-\bar t_0)\triangle}\bar y _0-  y_0
+ \int^{t_0}_{\bar  t_0} e^{(t_0-s)\triangle} \bar u^{\bar M}_{\bar t_0,\bar y_0}(s)ds\right\|\\
\ns=&\ds  \left\|e^{(t_0-\bar t_0)\triangle}\bar y _0- y_0
+ \int^{ t_0}_{\bar t_0} e^{(  t_0-s)\triangle} \bar u^{\bar M}_{\bar t_0,\bar y_0}(s)ds\right\|\\
 \ns\le&\ds  \left\|e^{(\bar t_0-t_0)\triangle}-I\right\|\cdot\|\bar y_0 \|
 +\| y _0-\bar y_0\|
+ \int^{ t_0}_{\bar t_0}\left\| e^{( t_0-s)\triangle}\right\|\cdot \left\|\bar u^{\bar M}_{\bar t_0,\bar y_0}(s)\right\| ds\\
 \ns\le&\ds\| y _0-\bar y_0\|
 +\left\|e^{(\bar t_0-t_0)\triangle}-I\right\|\cdot\|\bar y_0\|
+\bar M ( t_0-\bar{t_0}).
\ea$$

\noindent From the above-mentioned two estimates, we derive that
\bb\label{psiC2}\ba{c}
 \ns~~~~\left \|\bar\psi^M_{t_0,   y_0}(t_0)-\bar\psi^{\bar M}_{\bar
t_0,\bar y_0}(\bar t_0)\right\|
\le 2\| y _0-\bar y_0\|+2\bar M (\bar t_0-{t_0})+ 2\left\|e^{(
t_0-\bar t_0)\triangle}-I\right\|\cdot \|\bar y_0\|\\
\nm\ds

 + \left( 2\left\|\bar y^{\bar M}_{\bar t_0,\bar
y_0}( t_0)\right\|+\left\|\bar\psi^{\bar M}_{\bar t_0, \bar
y_0}(\bar t_0)\right\|+2\|z_d\|  \right)\frac{|M-\bar M|}{
\overline{ M}}
+\left \|\bar\psi^{ \bar M}_{\bar
  t_0,\bar   y_0}(t_0)-\bar\psi^{ \bar M}_{\bar
  t_0,\bar   y_0}( \bar t_0)\right\|
  \ea\ee%
By the same argument used to get (\ref{psibounded1}) (notice  the
continuity of $\bar\psi^{ \bar M}_{\bar
  t_0,\bar   y_0}( \cdot)$),  we can find a $\rho_2$ with
$\ds 0<\rho_2<\frac{T-\bar t_0}{2}\wedge\frac{\bar M}{2}$,
such that  for each triplet  $( t_0,y_0,M)\in [ \bar t_0 ,\bar
t_0+\rho_2]\times B(\bar y_0),\rho_2)\times [\bar M-\rho_2, \bar M+\rho_2]$,
\bb\label{psibounded1-A}
 \left \|\bar\psi^M_{t_0,   y_0}(t_0)-\bar\psi^{\bar M}_{\bar t_0,\bar y_0}(\bar t_0)\right\|
\le\ds\frac{1}{2} \left\|\bar\psi^{\bar M}_{\bar t_0,\bar y_0}(\bar
t_0)\right\|.  \ee
Now we set $ \bar\rho=\rho_1\wedge\rho_2$. Then the first inequality
of  (\ref{LOWERBOUNDED1}) follows from (\ref{psibounded1}) and
(\ref{psibounded1-A}).

\bigskip

\noindent {\it Step 3. When $\bar M=0$, there is $\bar\rho>0$ such that
for each $t_0\in [(\bar t_0-\bar \rho)^+, \bar t_0+\bar \rho]$, each
$y_0\in B(\bar y_0,\bar \rho)$ and each $M\in [0,\bar \rho]$,
\begin{eqnarray}\label{LOWERBOUNDED1-B}
\left\|\chi_{\omega}\bar\psi^{M}_{t_0,y_0}(t_0)\right\|\geq\frac{1}{2}
\left\|\chi_{\omega}\bar\psi^{0}_{\bar t_0,\bar y_0}(\bar
t_0)\right\|>0.
\end{eqnarray}  }
The second inequality in (\ref{LOWERBOUNDED1-B}) follows from (\ref{uniqueFHL}). The remainder is to show the first one.  The following two
inequalities can be checked by direct computation:
$$\ba{ll}
 \ns&\ds\left\|\bar\psi^M_{t_0, y_0}(t_0)- \bar\psi^0_{ t_0, \bar y_0}(t_0)  \right\|
 =\left\|e^{( T-\bar t_0)\triangle}\left[\bar\psi^M_{t_0,y_0}(T)- \bar\psi^0_{ t_0, \bar y_0}(T)\right]\right\|\\
 \ns=&\ds
 \left\|e^{( T-\bar t_0)\triangle}\left[\bar y^M_{t_0,y_0}(T)- \bar y^0_{ t_0, \bar y_0}(T)\right]\right\|
 \le\left\|\bar y^M_{t_0,y_0}(T)- \bar y^0_{ t_0, \bar y_0}(T)\right\|\\
 \ns=&\ds\left\|e^{( T-  t_0)\triangle}(y_0-\bar y_0)
 + \int^{T}_{t_0} e^{(T-s)\triangle} \bar u^{  M}_{ t_0,  y_0}(s)ds\right\|\\
 \ns\leq&\left\|y_0-\bar y_0\right\|
 +M(T- t_0)
 \ea$$
\noindent and
$$\ba{ll}
 \ns&\left\|\bar\psi^0_{ t_0, \bar y_0}(t_0) -\bar\psi^0_{\bar t_0, \bar y_0}(\bar
 t_0)\right\|\\
 \ns=&
 \left\|\left[e^{2( T-  t_0)\triangle}-e^{2( T-\bar
 t_0)\triangle}\right]\bar y_0
 +
 \left[e^{( T-  t_0)\triangle}-e^{ ( T-\bar
 t_0)\triangle}\right]z_d\right\|\\
 \ns\le&
 \left\| e^{2( | t_0-\bar t_0|)\triangle}-I \right\|\|\bar y_0\|
 +\left[e^{( |t_0-\bar t_0|)\triangle}-I\right]\|z_d\|.
\ea $$
From these,  we deduce that
\bb\label{psiC3} \ba{ll}
  \ns&\left\|\bar\psi^M_{t_0, y_0}(t_0)- \bar\psi^0_{ \bar t_0, \bar y_0}(\bar t_0)\right\|\\
  \ns\le&\ds\left\| \bar\psi^M_{t_0, y_0}(t_0)- \bar\psi^0_{ t_0, \bar y_0}(t_0)\right\|
  +\left\|\bar\psi^0_{ t_0, \bar y_0}(t_0)
  -\bar\psi^0_{\bar t_0, \bar y_0}(\bar  t_0)\right\|\\
  \ns\le&\ds\|y_0-\bar y_0\| +M(T- t_0)
   +
 \left\|e^{2( | t_0-\bar t_0|)\triangle}-I \right\|\|\bar y_0\|
 +\left[e^{( |t_0-\bar t_0|)\triangle}-I\right]\|z_d\|.
\ea \ee
Clearly,  the right hand side of (\ref{psiC3}) is continuous with respect to
$t_0,y_0$ and $M$. This, together with the second inequality of
(\ref{LOWERBOUNDED1-B}), yields that  there exists $ \bar \rho\in
(0, \ds\frac{T-\bar t_0}{2})$ such that for each $( t_0,y_0,M)\in
[(\bar t_0-\bar \rho )^+,\bar t_0+\bar \rho]\times
 B(\bar y_0,\bar\rho)\times [0,\bar \rho]$,
\begin{eqnarray*}
\left\|\bar\psi^M_{t_0, y_0}(t_0)- \bar\psi^0_{ \bar t_0, \bar
y_0}(\bar t_0)\right\|\leq \ds\frac{1}{2} \left\|\bar\psi^{0}_{\bar
t_0,\bar y_0}(\bar t_0)\right\|,
\end{eqnarray*}
from which, the first inequality in (\ref{LOWERBOUNDED1-B}) follows
at once.\\

\noindent{\it Step 4. The required Lipschitz continuity of the map
$\cN$}

Clearly, we can take the same constant $\bar \rho$ in Step 2 and Step 3 such that (\ref{LOWERBOUNDED1}) and (\ref{LOWERBOUNDED1-B}) stand. When
$\bar M=0$, it follows from (\ref{estimate2}), (\ref{estimate2-1})
and (\ref{LOWERBOUNDED1-B}) that the map $\cN(t_0,\cdot,\cdot)$ is
Lipschitz continuous over $B(y_0, \bar \rho)\times [(\bar M-\bar \rho)^+, \bar M+\bar \rho]$
uniformly with respect to $t_0\in [(\bar t_0-\bar \rho)^+, \bar t_0+\bar \rho]$; When $\bar M>0$, the same conclusion follows
from (\ref{estimate2}), (\ref{estimate2-1}) and
(\ref{LOWERBOUNDED1}).

\bigskip

\noindent $(ii)$ Fix a $\bar y_0\in L^2(\om)$. Let $(\bar M, \bar t_0)\in [0,\infty)\times [0,T)$.
Since $\|\chi_\omega\bar\psi^M_{\bar t_0,
\bar y_0}(\bar t_0)\|\neq 0$ (see (\ref{uniqueFHL})), the continuity of the map $(t_0,M)\rightarrow
\cN(t_0,\bar y_0, M)$ at $(\bar t_0,\bar M)$ follows from the continuity of the map
$(t_0, M)\rightarrow
\bar\psi^M_{t_0,
\bar y_0}(t_0)$ at $(\bar t_0,\bar M)$. When  $\bar
M>0$, the continuity of the map
$(t_0, M)\rightarrow
\bar\psi^M_{t_0,
\bar y_0}(t_0)$ at $(\bar t_0,\bar M)$ follows from (\ref{psiC1}) and  (\ref{psiC2}). When  $\bar M=0$, the continuity of this map at $(\bar t_0,\bar M)$ follows from
(\ref{psiC3}). Thus, $\cN(\cdot,\bar y_0, \cdot)$ is continuous over $[0,T)\times [0,\infty)$.

In summary, we complete the proof.
\endpf

\bigskip

\noindent{\bf Proof of Proposition~\ref{LIPOFF}.} $(i)$ Let $(\bar t_0,\bar y_0)\in [0,T)\times L^2(\Omega)$. By the
definition of maps $\cN$ and $F$ (see (\ref{feedback}) and
(\ref{MAPCN})), we see that
\begin{eqnarray}\label{HK10}F(t_0,y_0)=\cN(t_0,y_0,N(t_0,y_0))\;\; {\rm for ~all}~(t_0,y_0)\in[0,T)\times L^2(\om).
\end{eqnarray}
According to  Lemma  \ref{Lip2}, there are $\bar\rho_1>0$ and
$C_1>0$ such that when  $(t_0,y_0^1,M_1)$ and  $(t_0,y_0^2,M_2)$
belong to $[(\bar t_0-\bar\rho_1)^+,\bar t_0+\bar\rho_1]\times
B(\bar y_0,\bar\rho_1)\times [(\bar M-\bar\rho_1)^+,\bar
M+\bar\rho_1]$,
\bb\label{LP}\|\cN(t_0,y_0^1,M_1)-\cN(t_0,y_0^2,M_2)\|\le
C_1\Bigr(\|y_0^1-y_0^2\|+|M_1-M_2|\Big).\ee
According to Lemma \ref{Lip1}, there are $\bar\rho_2>0$ and $C_2>0$
such that
$$|N(t_0,y_0^1)-N(t_0,y_0^2)|\le C_2 \|y_0^1-y_0^2\|,$$
for all  $(t_0,y_0^1)$ and $(t_0,y_0^2)$ belong to $[(\bar
t_0-\bar\delta)^+,\bar t_0+\bar\delta]\times B(\bar
y_0,\bar\rho_2)$, where $\bar \delta=(T-\bar t_0)/2$.

Let $\bar\rho=
\min\{\bar\rho_1,\bar\rho_2,\ds\frac{\bar\rho_1}{2C_2},\bar\delta\}$.
Then it follows from the above inequality that
$$|N(t_0,y_0^1 )-N(t_0,\bar y_0^2)|\le 2 C_2  \bar\rho\le\bar\rho_1\;\;\mbox{for all}\; y_0^1,y_0^2 \in B(\bar y_0,\bar\rho)\; \mbox{and}\;  t_0\in[(\bar t_0-\bar\rho )^+,\bar t_0+\bar\rho ].$$
This, along with (\ref{LP}), indicates that
$$\ba{l}
~~~\|F(t_0,y_0^1)-F(t_0,y_0^2)\|\\
\ns=\|\cN(t_0,y_0^1,N(t_0,y_0^1))-\cN(t_0,y_0^2,N(t_0,y_0^2))\|\\
\ns\le
C_1\Bigr(\|y_0^1-y_0^2\| +|N(t_0,y_0^1)-N(t_0,y_0^2)|\Bigr)\\
\ns\le
C_1\Bigr(\|y_0^1-y_0^2\|+C_2\|y_0^1-y_0^2\|\Bigr)=C_1(1+C_2)\|y_0^1-y_0^2\|
\ea.$$
for all $y_0^1,y_0^2 \in B(\bar y_0,\bar\rho)$ and $ t_0\in[(\bar
t_0-\bar\rho )^+,\bar t_0+\bar\rho ]$. The desired Lipschitz
continuity follows from the above inequality at once.

\noindent$(ii)$  Let $\bar y_0\in L^2(\om)$. Since $F(t_0,\bar y_0)=\cN(t_0,\bar y_0,N(t_0,\bar y_0))$ for all $t_0\in [0,T)$ (see (\ref{HK10})),
the desired continuity of $F(\cdot, \bar y_0)$  follows directly from the continuity of
$N(\cdot,\bar y_0)$  and $\cN(\cdot,\bar y_0,\cdot)$.

In summary, we complete the proof.\endpf

\subsection{Proof of Theorem~\ref{feedbacktheorem}}

\noindent{\bf Proof of Theorem~\ref{feedbacktheorem}.}  Let
$(t_0, y_0)\in[0,T)\times L^2({\Omega})$. By
theorem~\ref{theorem3.9},  $\bar y^{N(t_0,y_0)}_{t_0,y_0}(\cdot)$ is
the unique solution to Equation  (\ref{feedbacksystem}), i.e.,
$y_F(\cdot;t_0,y_0)=\bar y^{N(t_0,y_0)}_{t_0,y_0}(\cdot)$ over
$[t_0,T)$.
 Then, by (\ref{feedback}), (\ref{DG2}), (\ref{DG1})  and (\ref{op}), we see that
$$
F\Bigr(t,y_F(t;t_0,y_0)\Bigl)=F\left(t, \bar
y^{N(t_0,y_0)}_{t_0,y_0}(t)\right) =N(t_0,y_0)\frac
{{\chi}_{\omega}\bar\psi^{N(t_0,y_0)}_{t_0,y_0}(t)}{\|{\chi}_{\omega}\bar\psi^{N(t_0,y_0)}_{t_0,y_0}(t)\|}
=\bar u^{N(t_0,y_0)}_{t_0,y_0}(t),\q\forall~t\in(t_0,T).$$
Since  $\bar u^{N(t_0,y_0)}_{t_0,y_0}(\cdot)$ is optimal norm
control for Problem $(NP)_{t_0,y_0}$ (see Lemma\ref{feedback-1}),
the above equality implies that $F(\cdot;y_F(\cdot;t_0,y_0)$ is the
optimal control to  $(NP)_{t_0,y_0}$. This completes the proof.
\endpf

\end{document}